\documentclass[preprint,12pt]{elsarticle}

\usepackage{graphicx}
\usepackage[T1]{fontenc}
\usepackage[utf8]{inputenc}
\usepackage{textcomp}
\usepackage[]{amsmath, amssymb}
\usepackage{fancyvrb}
\usepackage{caption}
\usepackage{float}
\usepackage{url}
\usepackage{pdfpages}
\usepackage{listings}
\usepackage{amsmath, amssymb}
\usepackage{amsthm}
\usepackage{grffile}
\usepackage{multirow}
\usepackage{pgfplots}
\usepackage{silence}
\usepackage{subcaption}
\usepackage{bm}

\DefineVerbatimEnvironment{code}{Verbatim}{fontsize=\small, xleftmargin=5mm}
\newcommand{\norm}[1]{\|#1\|}
\newcommand{\snorm}[1]{\|#1\|^2}

\newcommand{\del}{\partial}
\newcommand{\pder}[2]{\frac{\partial #1}{\partial #2}}

\newcommand{\bv}[1]{\mathbf{#1}}

\definecolor{light-gray}{gray}{0.8}

\theoremstyle{plain}
\newtheorem{theorem}{Theorem}
\newtheorem{lemma}[theorem]{Lemma}

\theoremstyle{definition}
\newtheorem{defi}{Definition}

\theoremstyle{remark}

\theoremstyle{remark}
\newtheorem{remark}{Remark}

\newlength\figureheight
\newlength\figurewidth

\usepackage{xcolor}
\usepackage{hyperref} 

\journal{Journal of Computational Physics}

\begin{document}

\begin{frontmatter}

	\title{Acoustic shape optimization using energy stable curvilinear finite differences}

	\author{Gustav Eriksson\corref{cor1}}
	\ead{gustav.eriksson@it.uu.se}
	\author{Vidar Stiernstr\"{o}m\corref{cor2}}
	\ead{cstierns@stanford.edu}

	\cortext[cor1]{Corresponding author.}
	\cortext[cor2]{Now at the Department of Geophysics, Stanford University, Stanford, USA.}
	\address{Department of Information Technology, Uppsala University, PO Box 337, S-751 05 Uppsala, Sweden}

	\begin{abstract}
		A gradient-based method for shape optimization problems constrained by the acoustic wave equation is presented. The method makes use of high-order accurate finite differences with summation-by-parts properties on multiblock curvilinear grids to discretize in space. Representing the design domain through a coordinate mapping from a reference domain, the design shape is obtained by inverting for the discretized coordinate map. The adjoint state framework is employed to efficiently compute the gradient of the loss functional. Using the summation-by-parts properties of the finite difference discretization, we prove stability and dual consistency for the semi-discrete forward and adjoint problems. Numerical experiments verify the accuracy of the finite difference scheme and demonstrate the capabilities of the shape optimization method on two model problems with real-world relevance.
	\end{abstract}

	\begin{keyword}
		shape optimization \sep acoustic wave equation \sep finite differences \sep summation-by-parts \sep adjoint method
	\end{keyword}
\end{frontmatter}
\section{Introduction}\label{sec: intro}
Partial differential equation (PDE) constrained shape optimization is a highly important topic for applications in research and engineering. Together with topology optimization, it constitutes a central tool for computer-aided optimal design problems. To obtain an efficient algorithm for PDE-constrained optimization problems, gradient-based methods are usually employed. In the context of PDE-constrained optimization, the adjoint framework has proven to be a very efficient approach for computing gradients of the loss functional, especially if the number of design variables is large \cite{Giannakoglou2008,plessix2006review}. Usually, one has to choose whether to apply the adjoint method on the continuous or discretized equations. In the first approach, \emph{optimize-then-discretize} (OD), the continuous gradient and the forward and adjoint equations are discretized with methods of choice, potentially unrelated to one another. OD is typically easier to analyze and implement since it does not have to consider the particularities of the discretization methods. However, the computed gradient will in general not be the exact gradient of the discrete loss functional, which can lead to convergence issues in the optimization \cite{Glowinski1998}. The alternative approach, \emph{discretize-then-optimize} (DO), computes the exact gradient (up to round-off) with respect to the discrete forward and adjoint problems. However, depending on how the forward problem is discretized, the DO approach can lead to stability issues in the discrete adjoint problem \cite{collis2002analysis}. In this work, we discretize space using a finite difference scheme that is \emph{dual-consistent}. Disregarding the discretization of time, this means that the scheme from the DO approach is a stable and consistent approximation of the scheme from the OD approach, such that the two approaches are equivalent. {Dual-consistent discretizations and superconvergent approximate functionals from PDEs are well-known concepts in many disciplines where integral quantities are of interest \cite{babuska,pierce_adjoint,lu2005posteriori}. See also \cite{doi:10.1137/100790987,berg_nordstrom_2012} for recent examples from the finite difference community.}

Given that the computation of directional derivatives with respect to the design geometry is a key part of gradient-based shape optimization, the question of how the geometry is represented as a variable is important. Various approaches exist in the literature. Perhaps the most obvious and standard approach is based on domain transformations, where diffeomorphisms are used to map coordinates between the physical domain subject to optimization and a fixed reference domain \cite{doi:10.1137/1.9780898719826}. The coordinate mapping then provides a straightforward way to express integrals and derivatives in the physical domain in terms of integrals and derivatives in a fixed reference domain, greatly simplifying the computation of the gradient. Although the simplicity of this approach is highly attractive, it can lead to robustness issues, such that re-meshing is required to obtain an accurate solution \cite{berggren2023}. An alternative approach is to discretize the problem using a so-called fictitious-domain method, where the domain of interest is embedded in a larger domain of simpler shape \cite{GLOWINSKI1994283}. In \cite{https://doi.org/10.1002/nme.5621} a fictitious-domain method is used together with level-set functions and CutFEM to solve a shape optimization problem constrained by the Helmholtz equation. In the present study, we discretize the PDE using finite differences defined on boundary-conforming structured grids. For this reason, we take the first approach and represent the geometry through a coordinate mapping between the physical domain and a rectangular reference domain. As demonstrated by the numerical results in Section \ref{sec: num_exp}, the issue of poor mesh quality and re-meshing can be avoided by the use of regularization.

In the present study, we consider shape optimization problems constrained by the acoustic wave equation. Due to the hyperbolic nature of this PDE, high-order finite difference methods are well-suited discretization methods in terms of accuracy and efficiency \cite{Kreiss1972,KREISS1974}. One way to obtain robust and provably stable high-order finite difference methods is to use difference stencils with summation-by-parts (SBP) properties \cite{Svard2014,DelReyFernandez2014}. SBP finite differences have been used in multiple studies of the second-order wave equation in the past \cite{Mattsson2009,Mattsson2012}. With careful treatment of boundary conditions, these methods allow for semi-discrete stability proofs and energy estimates that are equivalent to the continuous energy balance equations. Methods for imposing boundary conditions include, e.g., the simultaneous-approximation-term method (SBP-SAT) \cite{STIERNSTROM2023112376,Almquist2020}, the projection method (SBP-P) \cite{Mattsson2006,Eriksson2023}, the ghost-point method (SBP-GP) \cite{Wang2019}, and a combination of SAT and the projection method (SBP-P-SAT), recently presented in \cite{Eriksson2023}. To the best of our knowledge, SBP finite differences employed to solve acoustic shape optimization problems have not yet been presented in the literature.

Although the shape optimization methodology developed in the present study focuses on acoustic waves it is also applicable to other types of wave equations. In addition, from a mathematical point of view, the method closely resembles the work in \cite{doi:10.1190/geo2022-0195.1}, where an SBP-SAT method is used for seismic full waveform inversion. The difference essentially lies in what the coefficients in the discretized equations represent. Here they are metric coefficients derived from the coordinate transformation between the physical and reference domains, while in \cite{doi:10.1190/geo2022-0195.1} they are unknown material properties. Indeed, the metric coefficients may be viewed as the material properties in an anisotropic wave equation on a fixed domain. {Other studies where SBP difference methods have been employed for adjoint-based optimization include aerodynamic shape optimization \cite{hicken_aerodyn}, optimization of turbulent flows \cite{kord_capecelatro_2023} and optimization of gas networks \cite{hossbach_et_al_2022}.}

The paper is developed as follows: In Section \ref{sec: prob_sec} the model problem is presented. Then, in Section \ref{sec: forward_problem}, we analyze the forward problem and present the SBP finite difference discretization. In Section \ref{sec: the_opt_prob} the optimization problem is considered, including the derivation of the semi-discrete adjoint problem and corresponding gradient expression. We evaluate the performance of the method using three numerical experiments in Section \ref{sec: num_exp}. Finally, the study is concluded in Section \ref{sec: concl}.

\section{Problem setup}
\label{sec: prob_sec}
The general type of shape optimization problems considered in this paper are of the form
\begin{subequations}
\label{eq: cont_min_prob}
\begin{equation}
		\min_{p} \mathtt{J}(u,p), \quad \text{such that} \\
\end{equation}
\begin{equation}
	\label{eq: cont_wave_eq_gen}
	\begin{alignedat}{4}
		u_{tt} &=  c^2 \Delta u + F(\bv{x},t), \quad &&\bv{x} \in \Omega_p, \quad &&t \in [0,T], \\
		L u &= g(\bv{x},t), &&\bv{x} \in \del \Omega_p, &&t \in [0,T], \\
		u &= 0, \quad u_t =0,  &&\bv{x} \in \Omega_p, &&t = 0, \\
	\end{alignedat}
\end{equation}
\end{subequations}
where $\mathtt{J}(u,p)$ is the loss functional, $F(\bv{x},t)$ is a forcing function, $c$ is the wave speed, and the linear operator $L$ together with boundary data $g(\bv{x},t)$ defines the boundary conditions. Essentially, the problem consists of finding the control variable $p$ determining the shape of the domain $\Omega_p \in \mathbb{R}^2$ such that $\mathtt{J}(u,p)$ is minimized while $u$ satisfies the acoustic wave equation \eqref{eq: cont_wave_eq_gen}. We use the subscript $p$ to indicate the domain's dependency on the control parameter.

To make the analysis easier to follow, consider the model problem given by 
\begin{subequations}
	\label{eq: cont_min_prob_multiblock}
	\begin{equation}
		\min_{p} {\mathtt{J}(u,p) = \int _0^T} (u^+(\bv{x}_r,t) - u_d(t))^2 \: dt, \quad \text{such that} \\
	\end{equation}
	\begin{equation}
	\label{eq: cont_wave_eq_multiblock}
	\begin{array}{lll}
		u^+_{tt} = c^2 \Delta u^+ + f(t) \hat{\delta}(\bv{x} - \bv{x_s}), 	&\bv{x} \in \Omega^+, 		&t \in [0,T], \\
		u^-_{tt} = c^2 \Delta u^-, 									&\bv{x} \in \Omega^-_p, 	&t \in [0,T], \\
		u^\pm_t + c \bv{n}^\pm \cdot \nabla u^\pm = 0, \quad &\bv{x} \in \del \Omega_p^{(\pm,w,e)},  &t \in [0,T], \\
		\bv{n}^- \cdot \nabla u^- = 0, \quad &\bv{x} \in \del \Omega_p^{(-,s)}, \quad &t \in [0,T], \\
		u^+ = 0, \quad &\bv{x} \in \del \Omega_p^{(+,n)}, \quad &t \in [0,T], \\
		u^+ - u^- = 0,												& \bv{x} \in \Gamma_I, &t \in [0,T], \\
		\bv{n}^+ \cdot \nabla u^+ + \bv{n}^- \cdot \nabla u^- = 0,			& \bv{x} \in \Gamma_I, &t \in [0,T], \\
		u^\pm = 0,  \quad u^\pm_t = 0, 										&\bv{x} \in \Omega_p, 		&t = 0,
	\end{array}
	\end{equation}
\end{subequations}
with domain $\Omega_p$ as depicted in Figure \ref{fig: lake_bc}. We assume that the top half of the domain is known and therefore split it into two blocks $\Omega_p = \Omega^+ \cup \Omega^-_p$, where only $\Omega^-_p$ is subject to the optimization. Superscripts $+$ and $-$ are used to indicate variables defined on the respective domains. This model problem can be thought of as determining the bathymetry in a segment of a lake or sea, given a time series of recorded pressure $u_d$ at a receiver located in $\bv{x_r}$ originating from an acoustic point source at $\bv{x_s}$, where $\hat{\delta}$ denotes the Dirac delta function. For simplicity, we also assume that we have one receiver and one source and that they are located in $\Omega^+$. Extensions to multiple sources and receivers in $\Omega^+$ follow straightforwardly. At the top boundary, a homogeneous Dirichlet boundary condition is prescribed, corresponding to a constant surface pressure, and at the curved bottom boundary we prescribe a fully reflecting homogeneous Neumann boundary condition. At the sides, first-order outflow boundary conditions \cite{engquist} are prescribed. For well-posedness, we further require continuity of the solutions $u^\pm$ and the flux $\bv{n}^\pm \cdot \nabla u^\pm$ across the interface $\Gamma_I = \Omega^+ \cap \Omega^-_p$, where $\bv{n}^\pm$ are the outwards-pointing normals of the respective domain.
\begin{figure}[!htbp]
	\centering
	\includegraphics[width=0.5\textwidth]{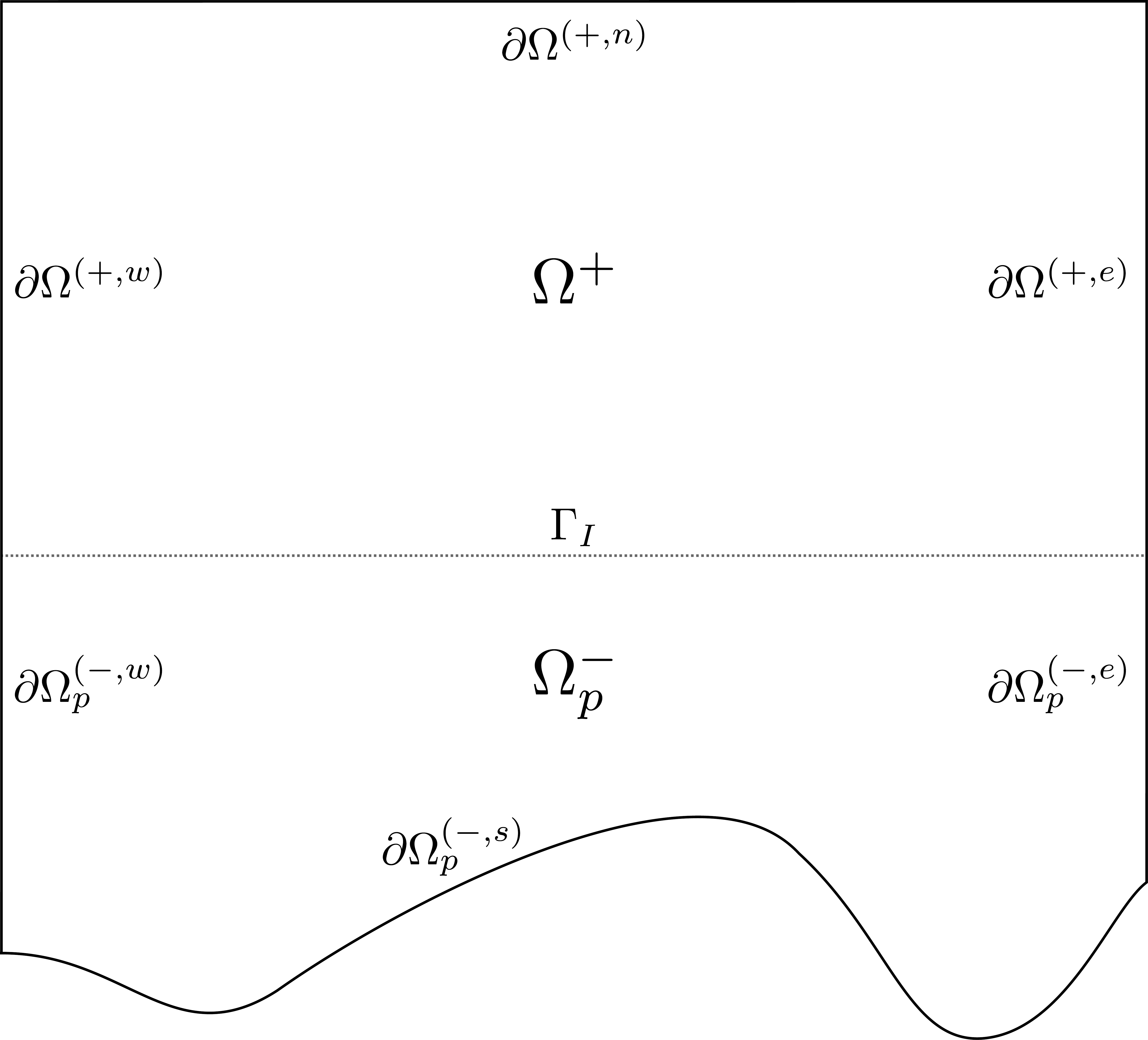}
	\caption{Boundary conditions of the bathymetry problem. The dotted line indicate the interface between $\Omega^+$ and $\Omega^-_p$.}
	\label{fig: lake_bc}
\end{figure}

\section{The forward problem}
\label{sec: forward_problem}
Initially, consider only the forward problem \eqref{eq: cont_wave_eq_multiblock} and its spatial discretization. In Section \ref{sec: the_opt_prob} we will return to the optimization problem. The acoustic wave equation on curvilinear or multiblock domains discretized using SBP finite differences has been treated on multiple occasions in the past, see e.g., \cite{Almquist2014, Wang2019, Almquist2019, Almquist2020, oreilly_petersson_2020, STIERNSTROM2023112376, Eriksson2023}. In this work, we formulate discretizations of the Laplace operator directly on the physical domains $\Omega^+$, and $\Omega^-_p$. As previously mentioned, on the non-rectangular domain $\Omega^-_p$ this is achieved through a coordinate map to a square reference domain. By incorporating the metric transformation, a process referred to as 'encapsulation' in \cite{alund_norstrom_2019}, the discrete Laplace operator satisfies SBP properties on  $\Omega^-_p$. Although this analysis is not new, it is reiterated here without proofs, since it is needed later in the context of optimizing for $p$. Using the discrete SBP Laplace operators we then enforce boundary and interface conditions through the SBP-P-SAT method. 

\subsection{Continuous analysis}
\label{sec: cont_analysis}
We begin by considering the well-posedness of \eqref{eq: cont_wave_eq_multiblock}. For simplicity, the wave speed $c$ is assumed to be constant in $\Omega_p$, but the extension to spatially variable $c(\bv{x})$ is straightforward. Let
\begin{equation}
	(u,v)_\Omega \equiv \int _{\Omega} u v \: d\bv{x}, \quad \snorm{u}_\Omega \equiv (u,u)_\Omega, \quad \text{and} \quad \langle u,v \rangle_{\partial \Omega} \equiv \int _{\partial \Omega} u v \: ds.
\end{equation}
{We begin by proving stability of \eqref{eq: cont_wave_eq_multiblock}, where it is sufficient to only consider the homogeneous problem, i.e. $f(t)$ is set to zero.} Multiplying the first equation in \eqref{eq: cont_wave_eq_multiblock} by $u^+_t$, the second equation by $u^-_t$, integrating over the domains, adding the results, and using Green's first identity leads to the energy equation
\begin{equation}
	\frac{dE}{dt} = 2 c^2 \langle u^+_t , \bv{n}^+ \cdot \nabla u^+ \rangle _{\partial \Omega^+} + 2 c^2 \langle u^-_t , \bv{n}^- \cdot \nabla u^- \rangle _{\partial \Omega^-_p},
\end{equation}
where $E$ is an energy given by
\begin{equation}
	E = \|u^+_t\|^2_{\Omega^+} + c^2  \|\nabla u^+\|^2_{\Omega^+_p} + \|u^-_t\|^2_{\Omega^-_p} + c^2 \|\nabla u^-\|^2_{\Omega^-_p} \geq 0.
\end{equation}
Inserting the boundary and interface conditions leads to
\begin{equation}
	\label{eq: two_block_enest}
	\frac{dE}{dt} = - 2 c \sum_{k = w,e} \langle u^+_t , u^+_t \rangle _{\partial \Omega^{(+,k)}} + \langle u^-_t , u^-_t \rangle _{\partial \Omega_p^{(-,k)}} \leq 0,
\end{equation}
which proves stability. From \eqref{eq: two_block_enest} it is clear that energy is dissipated through the east and west boundaries, due to the outflow boundary conditions.

The finite difference operators considered in the present study are defined on Cartesian grids. Therefore, for a general domain $\Omega^-_p$, we introduce a coordinate mapping from a reference domain $\tilde \Omega^- \in [0,1]^2$ to $\Omega^-_p$. Note that $\Omega^+$ is already rectangular, hence no coordinate transformation is needed in this part of the domain. Let
\begin{equation}
	\label{eq: coord_maps}
	x = x(\xi,\eta;p)	\quad \text{and} \quad y = y(\xi,\eta;p),
\end{equation}
define a smooth one-to-one coordinate mapping (diffeomorphism) from the reference domain $\xi,\eta \in \tilde \Omega^-$ to the physical domain $x,y \in \Omega_p^-$. Note that the functions $x$ and $y$ are parametrized by $p$. There are many ways one can derive the mappings \eqref{eq: coord_maps}. Here we use linear transfinite interpolation \cite{SMITH1982137}, which explicitly defines $x(\xi,\eta;p)$ and $y(\xi,\eta;p)$ given the coordinates of the boundary $\del \Omega_p^-$. Letting subscripts denote partial differentiation, the Jacobian determinant (or the area element) of the mapping is given by
\begin{equation}
	J = x_\xi y_\eta - x_\eta y_\xi.
\end{equation}
By the chain rule, derivatives with respect to $x$ and $y$ are given by
\begin{equation}
	\label{eq: cont_curv_dx_dy}
	\begin{alignedat}{2}
		u^-_x &= J^{-1} (y_\eta u^-_\xi - y_\xi u^-_\eta), \\
		u^-_y &= J^{-1} (-x_\eta u^-_\xi + x_\xi u^-_\eta).
	\end{alignedat}
\end{equation}
By extension, the Laplace operator in terms of derivatives in the reference domain is given by
\begin{equation}
	\label{eq: curv_laplace}
	\Delta u^- = J^{-1} ( (\alpha_1 u^-_\xi)_\xi + (\beta u^-_\xi)_\eta + (\beta u^-_\eta)_\xi + (\alpha_2 u^-_\eta)_\eta),
\end{equation}
where the metric coefficients are given by
\begin{equation}\label{eq: cont_curv_metric_coeff}
	\begin{alignedat}{2}
		\alpha_1 &= J^{-1} (x_\eta^2 + y_\eta^2), \\
		\beta &= -J^{-1} (x_\xi x_\eta + y_\xi y_\eta), \\
		\alpha_2 &=  J^{-1} (x_\xi^2 + y_\xi^2).
	\end{alignedat}
\end{equation}
Further, the transformed normal derivatives satisfy
\begin{equation}
		\label{eq: cont_curv_norms}
	 	\bv{n}^- \cdot \nabla u^- = 
	 	\begin{cases}
	 		\frac{\alpha_1 u^-_\xi + \beta u^-_\eta}{W_1}, \quad \xi, \eta \in  \del \tilde \Omega^{(-,w,e)}, \\
	 		\frac{\alpha_2 u^-_\eta + \beta u^-_\xi}{W_2}, \quad \xi, \eta \in  \del \tilde \Omega^{(-,s,n)},
	 	\end{cases}
\end{equation}
where $W_1$ and $W_2$ are boundary scaling factors given by
\begin{equation}
	W_1 = \sqrt{x_\xi^2 + y_\xi^2} \quad \text{and} \quad W_2 = \sqrt{x_\eta^2 + y_\eta^2}. 
\end{equation}
The inner products on the physical and reference domains are linked through the following equation:
\begin{equation}
	(u^-, v^-)_{\Omega_p^-} = (u^-, J v^-)_{\tilde \Omega^-}.
\end{equation} 
Similarly, the relation between boundary inner products on the physical and reference domains is given by
\begin{equation}
	\begin{alignedat}{2}
		\langle u^-, v^- \rangle_{\del \Omega_p^{(-,w,e)}} &= \langle u^-, W_1 v^- \rangle_{\del \tilde \Omega_p^{(-,w,e)}}, \\
		\langle u^-, v^- \rangle_{\del \Omega_p^{(-,s,n)}} &= \langle u^-, W_2 v^- \rangle_{\del \tilde \Omega_p^{(-,s,n)}}.
	\end{alignedat}
\end{equation}


\subsection{One-dimensional SBP operators}
We begin by presenting SBP finite difference operators in one spatial dimension, to be used as building blocks in Sections \ref{sec: discr_omega_plus} and \ref{sec: discr_omega_minus}. Let the interval $[x_l,x_r]$ be discretized by $m$ equidistant grid points $\bv{x} = [x_0,x_1,...,x_{m-1}]$, such that
\begin{equation}
	x_i = x_l + ih, \quad i = 0,1,...,m-1,
\end{equation}
where $h = \frac{x_r - x_l}{m-1}$ is the grid step size. SBP operators are associated with a norm matrix (or quadrature) $H$, defining an inner product and norm given by
\begin{equation}
	\label{eq: disc_1d_inprods}
	(\bv{u},\bv{v})_H = \bv{u}^\top H \bv{v} \quad \text{and} \quad \snorm{\bv{u}}_H = (\bv{u},\bv{u})_H, \quad \forall \bv{u},\bv{v} \in \mathbb{R}^m.
\end{equation}
{In this work we only consider diagonal-norm SBP operators, i.e. operators where $H$ is a diagonal matrix, for which the inner product in \eqref{eq: disc_1d_inprods} is simply the Euclidean dot product weighted by the diagonal of $H$}. {Later in the paper we shall use the same notation (inner product with a diagonal matrix subscript) to denote other discrete inner products defined equivalently, e.g., two-dimensional inner products.} We will also use boundary restriction vectors given by
\begin{equation}
	e_l^\top = [1,0,...,0] \quad \text{and} \quad e_r^\top = [0,...,0,1],
\end{equation}
and identity matrices $I_m$ of dimension $m \times m$.

To discretize the spatial derivatives we will use the first derivative SBP operators and the second derivative variable coefficient compatible SBP operators of interior orders four and six, first derived in \cite{strand1994, Mattsson2012}. The SBP properties of these operators are given in the following definitions:
\begin{defi}
	\label{def: sbp_D1}
	A difference operator $D_1 \approx \frac{\del}{\del x}$ is said to be a first derivative SBP operator if, for the norm $H$, the relation
	\begin{equation}
		H D_1 + D_1^\top H = -e_l e_l^\top + e_r e_r^\top,
	\end{equation}
	holds.
\end{defi}
\begin{defi}
	\label{def: sbp_var_D2}
	Let $D_1$ be a first derivative SBP operator and $\bv{c}$ be the restriction of a function $c(x)$ on the grid $\bv{x}$. A difference operator $D_2^{(\bv{c})} \approx \frac{\del}{\del x} c(x) \frac{\del}{\del x}$, is said to be a compatible second derivative SBP operator if, for the norm $H$, the relation
	\begin{equation}
		HD_2^{(\bv{c})} = -M^{(\bv{c})} - e_l e_l^\top \bv{c} d_l^\top + e_r e_r^\top \bv{c} d_r^\top,
	\end{equation}
	holds, where
	\begin{equation}
		M^{(\bv{c})} = D_1^\top H \bar{\bv{c}} D_1 + R^{(\bv{c})},
	\end{equation}
	with $R^{(\bv{c})}$ symmetric and semi-positive definite, $\bar{\bv{c}} = \mathtt{diag}(\bv{c})$, and 
	\begin{equation}
		d_l^\top = e_l^\top \hat{D}_1 \quad \text{and} \quad d_r^\top = e_r^\top \hat{D}_1,
	\end{equation}
	where the first and last rows of $\hat{D}_1$ approximates $\frac{\del}{\del x}$.
\end{defi}
Note that the remainder term $R^{(\bv{c})}$ is small (zero to the order of accuracy) \cite{Mattsson2012}. We also define
\begin{equation}
	D_2 \equiv D_2^{(\bv{1})} \quad \text{and} \quad R \equiv R^{(\bv{1})},
\end{equation}
where $\bv{1}$ is a vector of length $m$ with only ones.

\subsection{Two-dimensional operators on $\Omega^+$}
\label{sec: discr_omega_plus}
Since $\Omega^+$ is rectangular by construction, operators on $\Omega^+$ are directly obtained through tensor products of the one-dimensional operators. Let $m^+_x$ and $m^+_y$ denote the number of grid points in the $x$- and $y$-directions, respectively, and let $\bv{v}^+ \in \mathbb{R}^{N^+}$, where $N^+ = m^+_x m^+_y$, denote a column-major ordered solution vector. The two-dimensional SBP operators are constructed using tensor products as follows:
\begin{equation}
	\label{eq: disc_2d_tensor_prods_p}
	\begin{alignedat}{5}
		D^+_x &= (D_1 \otimes I_{m^+_y}), &&D^+_y = (I_{m^+_x} \otimes D_1), \\
		D^+_{xx} &= (D_2 \otimes I_{m^+_y}),  &&D^+_{yy} = (I_{m^+_x} \otimes D_2), \\
		H^+_x &= (H \otimes I_{m^+_y}), &&H^+_y = (I_{m^+_x} \otimes H), \\
		e^+_w &= (e_l^\top \otimes I_{m^+_y}), &&e^+_e = (e_r^\top \otimes I_{m^+_y}), \\
		e^+_s &= (I_{m^+_x} \otimes e_l^\top), &&e^+_n = (I_{m^+_x} \otimes e_r^\top), \\
		d^+_w &= -(d_l^\top \otimes I_{m^+_y}), &&d^+_e = (d_r^\top \otimes I_{m^+_y}), \\
		d^+_s &= -(I_{m^+_x} \otimes d_l^\top),\quad &&d^+_n = (I_{m^+_x} \otimes d_r^\top).
	\end{alignedat}	
\end{equation}
For notational clarity, we have used the same symbols for 1D operators in the $x$- and $y$-directions in \eqref{eq: disc_2d_tensor_prods_p}, although they are different in general (since they depend on the number of grid points and the grid spacing).

The discrete Laplace operator in $\Omega^+$ is
\begin{equation}
	D_L^+ = D_{xx}^+ + D_{yy}^+.
\end{equation}
{We also have the following norm matrices:
\begin{equation}
	\begin{alignedat}{2}
		H^+ &= H^+_x H^+_y, \\
		H^+_{w,e,s,n} &= H,
	\end{alignedat}
\end{equation}
with inner products $(\cdot,\cdot)_{H^+}$ and $(\cdot,\cdot)_{H^+_{w,e,s,n}}$ defined as in \eqref{eq: disc_1d_inprods}.} Using Definitions \ref{def: sbp_D1} and \ref{def: sbp_var_D2}, $D_L^+$ can be shown to satisfy {
\begin{equation}
	\label{eq: disc_laplace_sbp_p}
	\begin{alignedat}{2}
		(\bv{u},D_L^+ \bv{v})_{H^+} &= - (D^+_x \bv{u}, D^+_x \bv{v})_{H^+} - (D^+_y \bv{u}, D^+_y \bv{v})_{H^+} - (\bv{u},\bv{v})_{R^+} \\
		&+ \sum_{k = w,e,s,n} ( e^+_k \bv{u}, d^+_k \bv{v} )_{H^+_k}, \quad \forall \bv{u},\bv{v} \in \mathbb{R}^{N^+},
	\end{alignedat}
\end{equation}
}
where 
\begin{equation}
	(\bv{u},\bv{v})_{R^+} = \bv{u}^\top (R^+_{x} H^+_y + R^+_{y} H^+_x) \bv{v},
\end{equation}
with $R^+_{x} = (R \otimes I_{m^+_y})$ and $R^+_{y} = (I_{m^+_x} \otimes R)$. Note that $R^+_{x} \approx 0$, $R^+_{y} \approx 0$ (since $R \approx 0)$, and
\begin{equation}
	(\bv{u},\bv{v})_{R^+} = (\bv{v},\bv{u})_{R^+} \quad \text{and} \quad \snorm{\bv{u}}_{R^+} := (\bv{u},\bv{u})_{R^+} \geq 0, \quad \forall \bv{u},\bv{v} \in \mathbb{R}^{N^+}.
\end{equation}
\subsection{Curvilinear operators on $\Omega^-_p$}
\label{sec: discr_omega_minus}
We begin by discretizing $\tilde \Omega^-$ using $m^-_\xi$ and $m^-_\eta$ equidistant grid points in the $\xi$- and $\eta$-directions, respectively. As in $\Omega^+$, we let $\bv{v}^- \in \mathbb{R}^{N^-}$, where $N^- = m^-_\xi m^-_\eta$, denote a column-major ordered solution vector. Using tensor products, we have the following two-dimensional SBP operators in $\tilde \Omega^-$:
\begin{equation}
	\label{eq: disc_2d_tensor_prods_m}
	\begin{alignedat}{5}
		D^-_\xi &= (D_1 \otimes I_{m^-_\eta}), \quad &&D^-_\eta = (I_{m^-_\xi} \otimes D_1), \\
		\hat{D}^-_\xi &= (\hat{D}_1 \otimes I_{m^-_\eta}), \quad &&\hat{D}^-_\eta = (I_{m^-_\xi} \otimes \hat{D}_1), \\
		H^-_\xi &= (H \otimes I_{m^-_\eta}), &&H^-_\eta = (I_{m^-_\xi} \otimes H), \\
		e^-_w &= (e_l^\top \otimes I_{m^-_\eta}), &&e^-_e = (e_r^\top \otimes I_{m^-_\eta}), \\
		e^-_s &= (I_{m^-_\xi} \otimes e_l^\top), &&e^-_n = (I_{m^-_\xi} \otimes e_r^\top), \\
	\end{alignedat}	
\end{equation}
where $\hat{D}_1$ is the boundary derivative operator in Definition \ref{def: sbp_var_D2}. Note that with a general variable coefficients vector $\bv{c} \in \mathbb{R}^{N^-}$, tensor products can not be used to construct the two-dimensional variable coefficient operators $D^{(\bv{c})}_{\xi\xi}$ and $D^{(\bv{c})}_{\eta\eta}$. Instead, the one-dimensional operators are built line-by-line with the corresponding values of $\bv{c}$ and stitched together to form the two-dimensional operators. See \cite{Almquist2014} for more details on this. 

Next, we use the relations in Section \ref{sec: cont_analysis} to derive finite difference operators in $\Omega_p^-$. As previously mentioned, this derivation can be found with more detail in, e.g., \cite{Almquist2014, Almquist2020}. Since the mappings $x(\xi,\eta;p)$ and $y(\xi,\eta;p)$ are not generally known analytically, the metric derivatives $x_\xi$, $x_\eta$, $y_\xi$, and $y_\eta$ are computed using the first derivative SBP operators $D^-_\xi$ and $D^-_\eta$. If $\bv{x}$ and $\bv{y}$ are vectors containing the coordinates on the physical grid, we have the discrete metric coefficient diagonal matrices
\begin{equation}
	\label{eq: disc_metric_dervs}
	\begin{alignedat}{2}
		\bv{X}_\xi = \mathtt{diag}(D^-_\xi \bv{x}), \\
		\bv{X}_\eta = \mathtt{diag}(D^-_\eta \bv{x}), \\
		\bv{Y}_\xi = \mathtt{diag}(D^-_\xi \bv{y}), \\
		\bv{Y}_\eta = \mathtt{diag}(D^-_\eta \bv{y}),
	\end{alignedat}
\end{equation}
and
\begin{equation}
	\label{eq: disc_metric_coeffsa}
	\begin{alignedat}{2}
		\bv{J} &= \bv{X}_\xi \bv{Y}_\eta - \bv{X}_\eta \bv{Y}_\xi, \\
		\alpha_1 &= \bv{J}^{-1} (\bv{X}_\eta^2 + \bv{Y}_\eta^2), \\
		\beta &= -\bv{J}^{-1} (\bv{X}_\xi \bv{X}_\eta + \bv{Y}_\xi \bv{Y}_\eta), \\
		\alpha_2 &= \bv{J}^{-1} (\bv{X}_\xi^2 + \bv{Y}_\xi^2).
	\end{alignedat}
\end{equation}

Using \eqref{eq: cont_curv_dx_dy}, \eqref{eq: curv_laplace}, and \eqref{eq: cont_curv_norms}, we get the first derivative operators
\begin{equation}
	\begin{alignedat}{2}
		D_x^- &= J^{-1} (\bv{Y}_\eta D^-_\xi - \bv{Y}_\xi D^-_\eta), \\
		D_y^- &= J^{-1} (-\bv{X}_\eta D^-_\xi + \bv{X}_\xi D^-_\eta),
	\end{alignedat}
\end{equation}
the two-dimensional curvilinear Laplace operator
\begin{equation}
	D_L^- = J^{-1} (D_{\xi \xi}^{(\alpha_1)} + D^-_\eta \beta D^-_\xi + D^-_\xi \beta D^-_\eta + D_{\eta \eta} ^{(\alpha_2)}),
\end{equation}
and the discrete normal derivative operators
\begin{equation}
	\begin{alignedat}{2}
		d^-_w &= -e^-_w \bv{W}_2^{-1} (\alpha_1 e_w^{-\top} e^-_w \hat{D}^-_\xi + \beta e_w^{-\top} e^-_w D^-_\eta), \\
		d^-_e &= e^-_e \bv{W}_2^{-1} (\alpha_1 e_e^{-\top} e^-_e \hat{D}^-_\xi + \beta e_e^{-\top} e^-_e D^-_\eta), \\
		d^-_s &= -e^-_s \bv{W}_1^{-1} (\alpha_2 e_s^{-\top} e^-_s \hat{D}^-_\eta + \beta e_s^{-\top} e^-_s D^-_\xi), \\
		d^-_n &= e^-_n \bv{W}_1^{-1} (\alpha_2 e_n^{-\top} e^-_n \hat{D}^-_\eta + \beta e_n^{-\top} e^-_n D^-_\xi),
	\end{alignedat}
\end{equation}
where
\begin{equation}
	\label{disc_W1_W2_minus}
	\bv{W}_1 = \sqrt{\bv{X}_\xi^2 + \bv{Y}_\xi^2} \quad \text{and} \quad \bv{W}_2 = \sqrt{\bv{X}_\eta^2 + \bv{Y}_\eta^2}.
\end{equation}
{We also have the following norm matrices:
\begin{equation}
	\begin{alignedat}{2}
		H^- &= H^-_\xi H^-_\eta J, \\
		H^-_w &= H e_w \bv{W}_2 e_w^\top , \\
		H^-_e &= H e_e \bv{W}_2 e_e^\top , \\
		H^-_s &= H e_s \bv{W}_1 e_s^\top , \\
		H^-_n &= H e_n \bv{W}_1 e_n^\top ,
	\end{alignedat}
\end{equation}
with inner products $(\cdot,\cdot)_{H^-}$ and $(\cdot,\cdot)_{H^-_{w,e,s,n}}$ defined as in \eqref{eq: disc_1d_inprods}.} {Note that the power of two and square roots in \eqref{eq: disc_metric_coeffsa} and \eqref{disc_W1_W2_minus} are evaluated elementwise}. Using Definitions \ref{def: sbp_D1} and \ref{def: sbp_var_D2}, the Laplace operator $D_L^-$ can be shown to satisfy (see \cite{Almquist2020})
\begin{equation}
	\label{eq: disc_laplace_sbp_m}
	\begin{alignedat}{2}
		(\bv{u},D_L^- \bv{v})_{H^-} &= - (D^-_x \bv{u}, D^-_x \bv{v})_{H^-} - (D^-_y \bv{u}, D^-_y \bv{v})_{H^-} - (\bv{u},\bv{v})_{R^-} \\
		&+ \sum_{k = w,e,s,n} ( e^-_k \bv{u}, d^-_k \bv{v} )_{H^-_k}, \quad \forall \bv{u},\bv{v} \in \mathbb{R}^{N^-},
	\end{alignedat}
\end{equation}
where 
\begin{equation}
	(\bv{u},\bv{v})_{R^-} = \bv{u}^\top (R_\xi^{(\alpha_1)} H_\eta^- + R_\eta^{(\alpha_2)} H_\xi^-)\bv{v},
\end{equation}
with $R_\xi^{(\alpha_1)}$ and $R_\eta^{(\alpha_2)}$ created the same way as $D^{(\alpha_1)}_{\xi\xi}$ and $D^{(\alpha_2)}_{\eta\eta}$. Note that $R_\xi^{(\alpha_1)} \approx 0$, $R_\eta^{(\alpha_2)} \approx 0$ (since $R^{(\bv{c})} \approx 0)$, and
\begin{equation}
	(\bv{u},\bv{v})_{R^-} = (\bv{v},\bv{u})_{R^-} \quad \text{and} \quad \snorm{\bv{u}}_{R^-} := (\bv{u},\bv{u})_{R^-} \geq 0, \quad \forall \bv{u},\bv{v} \in \mathbb{R}^{N^-}.
\end{equation}
\remark{The two domains $\Omega^+$ and $\Omega^-_p$} are conforming at the interface $\Gamma_I$, and for convenience reasons, we shall use the same number of grid points on both sides of the interface, i.e. $m_{\Gamma_I} = m^+_x = m^-_\xi$. Therefore,
\begin{equation}
	(\bv{u},\bv{v})_{H_s^+} = (\bv{u},\bv{v})_{H^-_n}, \quad \forall \bv{u},\bv{v} \in \mathbb{R}^{m_{\Gamma_I}}.
\end{equation}
\subsection{SBP-P-SAT discretization}
We now return to the forward problem \eqref{eq: cont_wave_eq_multiblock}. By replacing all spatial derivatives in \eqref{eq: cont_wave_eq_multiblock} by their corresponding SBP operators, we obtain a constrained initial value problem given by,
\begin{equation}
	\label{eq: disc_const_ode}
	\begin{alignedat}{2}
		&\bv{v}^{+}_{tt} = c^2 D^{+}_L \bv{v}^{+} + f(t) \tilde{\bv{d}}_s, \quad &&t > 0, \\
		&\bv{v}^{-}_{tt} = c^2 D^{-}_L \bv{v}^{-},  &&t > 0, \\
		&e_{w,e}^{\pm} \bv{v}^{\pm}_t + c d^{\pm}_{w,e} \bv{v}^{\pm} = 0,  &&t > 0, \\
		&d^{-}_s \bv{v}^{-} = 0,  &&t > 0, \\
		&e^{+}_n \bv{v}^{+} = 0,  &&t > 0, \\
		&e^{+}_s \bv{v}^{+} - e^{-}_n \bv{v}^{-} = 0,  &&t > 0, \\
		&d^{+}_s \bv{v}^{+} + d^{-}_n \bv{v}^{-} = 0,  &&t > 0, \\
		& \bv{v}^\pm = 0, \quad \bv{v}_t^\pm = 0, && t = 0,
	\end{alignedat}
\end{equation}
where
\begin{equation}
	\label{eq: disc_point_source}
	\tilde{\bv{d}}_s = (H^+)^{-1} (\delta^{(s)}_x \otimes \delta^{(s)}_y).
\end{equation}
Here $\delta^{(s)}_x$ and $\delta^{(s)}_y$ are discrete one-dimensional point sources discretized as in \cite{PETERSSON2016532}, using the same number of moment conditions as the order of the SBP operators and no smoothness conditions. 

{We impose the boundary and interface conditions in \eqref{eq: disc_const_ode} using a combination of the SAT method \cite{Mattsson2008,Wang2018,Almquist2020} and the projection method \cite{Eriksson2023,ERIKSSON2023111907}. SATs are used to weakly impose the outflow and Neumann boundary conditions and to couple the fluxes across the interface, while the projection method strongly imposes the Dirichlet boundary conditions and continuity of the solution across the interface}. This choice is made to keep the projection operator as simple as possible while avoiding the so-called borrowing trick necessary for an SBP-SAT discretization \cite{Wang2016}. A hybrid SBP-P-SAT method for the second-order wave equation on {Cartesian} multiblock domains with non-conforming interfaces was presented in \cite{Eriksson2023}. Indeed, the scheme presented here is the same as in \cite{Eriksson2023} if the interpolation operators are replaced with identity matrices {and $\Omega^-_p$ is Cartesian}. More details on the projection method can be found in \cite{Olsson1995a,Olsson1995}. Let 
\begin{equation}
	\bv{w} = 
	\begin{bmatrix}
		\bv{v}^{+} \\ \bv{v}^{-}
	\end{bmatrix} \in \mathbb{R}^N,
\end{equation}
where $N = N^+ + N^-$, denote the global solution vector and 
\begin{equation}
	\bar H = 
	\begin{bmatrix}
		H^+ & 0 \\ 0 & H^-
	\end{bmatrix},
\end{equation}
a global norm matrix with associated inner product defined as in \eqref{eq: disc_1d_inprods}. A consistent SBP-P-SAT discretization is then given by
\begin{equation}
	\label{eq: disc_ODE_system}
	\begin{alignedat}{2}
		&\bv{w}_{tt} = D \bv{w} + E \bv{w}_t + f(t) \bv{d}_s, \quad  &&t > 0, \\
		&\bv{w} = 0, \quad \bv{w}_t = 0,  &&t = 0,
	\end{alignedat}
\end{equation}
where
\begin{equation}
\label{eq: disc_operator_specs}
\begin{alignedat}{2}
	D &= c^2
	P \left ( \begin{bmatrix}
		D_L^{+} & 0 \\
		0 & D_L^{-}
	\end{bmatrix} + SAT_{BC_1} + SAT_{IC} \right ) P, \\
	E &= c P SAT_{BC_2} P, \\
	SAT_{BC_1} &= -\bar H^{-1} \left (
	\sum_{k = w,e} \begin{bmatrix}
		e_k^{+ \top} H^{+}_k d_k^{+} & 0 \\
		0 & e_k^{- \top} H^{-}_k d_k^{-}
	\end{bmatrix} + \begin{bmatrix}
		0 & 0 \\
		0 & e_s^{- \top} H^{-}_s d_s^{-}
	\end{bmatrix} \right ), \\
	SAT_{BC_2} &= - \bar H^{-1}
	\sum_{k = w,e}  \begin{bmatrix}
		 e_k^{+ \top} H^{+}_k e_k^{+} & 0 \\
		0 & e_k^{-\top} H^{-}_k e_k^{-}
	\end{bmatrix}, \\
	SAT_{IC} &= -\bar H^{-1}
	\begin{bmatrix}
		e_s^{+ \top} H^{+}_s d^{+}_s & e_s^{+ \top} H^{+}_s d^{-}_n \\
		0 & 0
	\end{bmatrix}, \quad \text{and} \\
	\bv{d}_s &=
		\begin{bmatrix}
			\tilde{\bv{d}}_s \\ 0
		\end{bmatrix}.
\end{alignedat}
\end{equation}
The projection operator is given by
\begin{equation}
	P = I - \bar H^{-1} L^\top (L \bar H^{-1} L^\top)^{-1} L,
\end{equation}
where
\begin{equation}
	L = 
	\begin{bmatrix}
		e^{+}_n & 0 \\
		e^{+}_s & -e^{-}_n
		\end{bmatrix}.
\end{equation}
This corresponds to imposing the conditions
\begin{equation}
	\begin{alignedat}{2}
		&e_{w,e}^{\pm} \bv{v}^{\pm}_t + c d^{\pm}_{w,e} \bv{v}^{\pm} = 0,  \quad &&t > 0, \\
		&d^{-}_s \bv{v}^{-} = 0,  &&t > 0, \\
		&d^{+}_s \bv{v}^{+} + d^{-}_n \bv{v}^{-} = 0,  &&t > 0,
	\end{alignedat}
\end{equation}
weakly using the SAT method and the conditions
\begin{equation}
	\begin{alignedat}{2}
		&e^{+}_n \bv{v}^{+} = 0,  &&t > 0, \\
		&e^{+}_s \bv{v}^{+} - e^{-}_n \bv{v}^{-} = 0,  \quad &&t > 0,
	\end{alignedat}
\end{equation}
strongly using the projection method.

We now prove three Lemmas. The first one is with regards to the stability of the scheme \eqref{eq: disc_ODE_system} while the second and third are on its self-adjointness properties. Similar self-adjointness properties of SBP-SAT discretizations of the acoustic and elastic wave equation have been shown in e.g. \cite{almquist_dunham_2021, doi:10.1190/geo2022-0195.1}. However, to the best of our knowledge, the derivation is new for the SBP-P-SAT discretization presented herein.
\begin{lemma}
	\label{lemma: ode_stability}
	The SBP-P-SAT scheme \eqref{eq: disc_ODE_system} is stable.
\end{lemma}
\begin{proof}
	We prove stability using the energy method. Since data does not influence stability, we set $f(t) = 0$. Taking the inner product between $\bv{w_t}$ and \eqref{eq: disc_ODE_system} gives
	\begin{equation}
		\begin{alignedat}{2}
			\frac{1}{2} \frac{d}{dt} \snorm{\bv{w}_t}_{\bar{H}} &= c^2 (\tilde{\bv{v}}_t^{+}, D_L^{+} \tilde{\bv{v}}^{+})_{H^{+}} + c^2 (\tilde{\bv{v}}_t^{-}, D_L^{-} \tilde{\bv{v}}^{-})_{H^{-}} \\
			&- \sum_{k = w,e} (c \snorm{e_k^{+} \tilde{\bv{v}}_t^{+}}_{H^{+}_k} + c^2 ( e_k^{+} \tilde{\bv{v}}_t^{+}, d_k^{+} \tilde{\bv{v}}^{+} )_{H_k^{+}} \\
			&+ c \snorm{e_k^{-} \tilde{\bv{v}}_t^{-}}_{H^{-}_k} + c^2 ( e_k^{-} \tilde{\bv{v}}_t^{-}, d_k^{-} \tilde{\bv{v}}^{-} )_{H_k^{+}}) \\
			&- c^2 ( e_s^{-} \tilde{\bv{v}}_t^{-}, d_s^{-} \tilde{\bv{v}}^{-} )_{H_s^{-}} - c^2 ( e_s^{+} \tilde{\bv{v}}_t^{+}, d_s^{+} \tilde{\bv{v}}^{+} + d_n^{-} \tilde{\bv{v}}^{-} )_{H_s^{+}},
		\end{alignedat}
	\end{equation}
	where $P \bv{w} = \tilde{\bv{w}} = \begin{bmatrix} \tilde{\bv{v}}^+ \\ \tilde{\bv{v}}^- \end{bmatrix}$ denotes the projected solution vector. Using \eqref{eq: disc_laplace_sbp_p} and \eqref{eq: disc_laplace_sbp_m} and rearranging terms lead to
	\begin{equation}
		\label{eq: disc_en_eq_1}
		\begin{alignedat}{2}
			\frac{d}{dt} (E^+ + E^-) &= - 2c \sum_{k = w,e} \snorm{e_k^{-} \tilde{\bv{v}}_t^{-}}_{H^{-}_k} + \snorm{e_k^{+} \tilde{\bv{v}}_t^{+}}_{H^{+}_k} \\
			&+ 2 c^2 ( e_n^{-} \tilde{\bv{v}}_t^{-} - e_s^{+} \tilde{\bv{v}}_t^{+}, d_n^{-} \tilde{\bv{v}}^{-} )_{H_N^{(-)}} + 2 c^2 ( e_n^{+} \tilde{\bv{v}}_t^{+}, d_n^{+} \tilde{\bv{v}}^{+} )_{H_N^{+}},
		\end{alignedat}
	\end{equation}
	where 
	\begin{equation}
		\begin{alignedat}{2}
			E^\pm &= \snorm{\bv{v}^{\pm}_t}_{H^{\pm}} + c^2 \snorm{D^{\pm}_x \tilde{\bv{v}}^{\pm}}_{H^{\pm}} + c^2 \snorm{D^{\pm}_y \tilde{\bv{v}}^{\pm}}_{H^{\pm}} + c^2 \snorm{\tilde{\bv{v}}^{\pm}}_{R^\pm}.
		\end{alignedat}
	\end{equation}
	Since $L \tilde{\bv{w}} = L P \bv{w} = 0$ holds exactly by the definition of $P$, we have
	\begin{equation}
		e_n^{+} \tilde{\bv{v}}^{+} = 0, \quad \text{and} \quad e_n^{-} \tilde{\bv{v}}^{-} = e_s^{+} \tilde{\bv{v}}^{+}.
	\end{equation}
	Inserted into \eqref{eq: disc_en_eq_1} results in
	\begin{equation}
		\frac{d}{dt} (E^+ + E^-) = - 2c \sum_{k = w,e} \snorm{e_k^{-} \tilde{\bv{v}}_t^{-}}_{H^{-}_k} + \snorm{e_k^{+} \tilde{\bv{v}}_t^{+}}_{H^{+}_k} \leq 0,
	\end{equation}
	which is the discrete equivalent to the continuous energy equation \eqref{eq: two_block_enest} and proofs stability of \eqref{eq: disc_ODE_system}.
\end{proof}

\begin{lemma}
	\label{lemma: D_self_adj}
	The matrix $D$ in \eqref{eq: disc_operator_specs} is self-adjoint with respect to $\bar{H}$, i.e.
	\begin{equation}
		(\bv{u}, D \bv{v})_{\bar{H}} = (D \bv{u}, \bv{v})_{\bar{H}}, \quad \forall \bv{u},\bv{v} \in \mathbb{R}^N.
	\end{equation}
\end{lemma}

\begin{proof}
	Let $P \bv{u} = P \begin{bmatrix} {\bv{u}}^{+} \\ {\bv{u}}^{-} \end{bmatrix} = \tilde{\bv{u}} = \begin{bmatrix} \tilde{\bv{u}}^{+} \\ \tilde{\bv{u}}^{-} \end{bmatrix}$ and $P \bv{v} = P \begin{bmatrix} {\bv{v}}^{+} \\ {\bv{v}}^{-} \end{bmatrix} = \tilde{\bv{v}} = \begin{bmatrix} \tilde{\bv{v}}^{+} \\ \tilde{\bv{v}}^{-} \end{bmatrix}$, then
	\begin{equation}
		\begin{alignedat}{2}
			(\bv{u}, D \bv{v})_{\bar{H}} &= c^2 (\tilde{\bv{u}}^{+}, D_L^{+} \tilde{\bv{v}}^{+})_{H^{+}} + c^2 (\tilde{\bv{u}}^{-}, D_L^{-} \tilde{\bv{v}}^{-})_{H^{-}} \\
			&- c^2 \sum_{k = w,e} ( e_k^{+} \tilde{\bv{u}}^{+}, d_k^{+} \tilde{\bv{v}}^{+} )_{H_k^{+}} + ( e_k^{-} \tilde{\bv{u}}^{-}, d_k^{-} \tilde{\bv{v}}^{-} )_{H_k^{+}} \\
			&- c^2 ( e_s^{-} \tilde{\bv{u}}^{-}, d_s^{-} \tilde{\bv{v}}^{-} )_{H_s^{-}} - c^2 ( e_s^{+} \tilde{\bv{u}}^{+}, d_s^{+} \tilde{\bv{v}}^{+} + d_n^{-} \tilde{\bv{v}}^{-} )_{H_s^{+}}.
		\end{alignedat}
	\end{equation}
	Using \eqref{eq: disc_laplace_sbp_p} and \eqref{eq: disc_laplace_sbp_m} and rearranging terms lead to
	\begin{equation}
		\label{eq: disc_selfadj_1}
		\begin{alignedat}{2}
			(\bv{u}, D \bv{v})_{\bar{H}} &= - c^2 (D^{+}_x \tilde{\bv{u}}^{+}, D^{+}_x \tilde{\bv{v}}^{+})_{H^{+}} - c^2 (D^{+}_y \tilde{\bv{u}}^{+}, D^{+}_y \tilde{\bv{v}}^{+})_{H^{+}} \\
			&- c^2 (D^{-}_x \tilde{\bv{u}}^{-}, D^{-}_x \tilde{\bv{v}}^{-})_{H^{-}} - c^2 (D^{-}_y \tilde{\bv{u}}^{-}, D^{-}_y \tilde{\bv{v}}^{-})_{H^{-}} \\
			&- c^2 (\tilde{\bv{u}}^{+},\tilde{\bv{v}}^{+})_{R^{+}} - c^2 (\tilde{\bv{u}}^{-},\tilde{\bv{v}}^{-})_{R^{-}}\\
			& + c^2 ( e_n^{-} \tilde{\bv{u}}^{-} - e_s^{+} \bv{u}^{+}, d_n^{-} \tilde{\bv{v}}^{-} )_{H^{-}_n} + c^2 ( e_n^{+} \tilde{\bv{u}}^{+}, d_n^{+} \tilde{\bv{v}}^{+} )_{H^{+}_n}.
		\end{alignedat}
	\end{equation}
	Since $L P \bv{u} = 0$ hold exactly by the definition of $P$, we have
	\begin{equation}
		e_n^{+} \tilde{\bv{u}}^{+} = 0, \quad \text{and} \quad e_n^{-} \tilde{\bv{u}}^{-} = e_s^{+} \tilde{\bv{u}}^{+}.
	\end{equation}
	Inserted into \eqref{eq: disc_selfadj_1} results in
	\begin{equation}
	\label{eq: disc_selfadj_2}
		\begin{alignedat}{2}
			(\bv{u}, D \bv{v})_{\bar{H}} &= - c^2 (D^{+}_x \tilde{\bv{u}}^{+}, D^{+}_x \tilde{\bv{v}}^{+})_{H^{+}} - c^2 (D^{+}_y \tilde{\bv{u}}^{+}, D^{+}_y \tilde{\bv{v}}^{+})_{H^{+}} \\
			&- c^2(D^{-}_x \tilde{\bv{u}}^{-}, D^{-}_x \tilde{\bv{v}}^{-})_{H^{-}} - c^2 (D^{-}_y \tilde{\bv{u}}^{-}, D^{-}_y \tilde{\bv{v}}^{-})_{H^{-}} \\
			&- c^2(\tilde{\bv{u}}^{+},\tilde{\bv{v}}^{+})_{R^{+}} - c^2 (\tilde{\bv{u}}^{-},\tilde{\bv{v}}^{-})_{R^{-}}.
		\end{alignedat}
	\end{equation}
	Since all the terms on the right-hand side of \eqref{eq: disc_selfadj_2} are symmetric, i.e. we can swap $\bv{u}$ and $\bv{v}$ and obtain the same expression, we have
	\begin{equation}
		(\bv{u}, D \bv{v})_{\bar{H}} = (\bv{v}, D \bv{u})_{\bar{H}} = (D \bv{u}, \bv{v})_{\bar{H}},
	\end{equation}
	which proves the lemma.
\end{proof}

\begin{lemma}
	\label{lemma: E_self_adj}
	The matrix $E$ in \eqref{eq: disc_operator_specs} is self-adjoint with respect to $\bar{H}$, i.e.
	\begin{equation}
		(\bv{u}, E \bv{v})_{\bar{H}} = (E \bv{u}, \bv{v})_{\bar{H}}, \quad \forall \bv{u},\bv{v} \in \mathbb{R}^N.
	\end{equation}
\end{lemma}

\begin{proof}
Let $P \bv{u} = P \begin{bmatrix} {\bv{u}}^{+} \\ {\bv{u}}^{-} \end{bmatrix} = \tilde{\bv{u}} = \begin{bmatrix} \tilde{\bv{u}}^{+} \\ \tilde{\bv{u}}^{-} \end{bmatrix}$ and $P \bv{v} = P \begin{bmatrix} {\bv{v}}^{+} \\ {\bv{v}}^{-} \end{bmatrix} = \tilde{\bv{v}} = \begin{bmatrix} \tilde{\bv{v}}^{+} \\ \tilde{\bv{v}}^{-} \end{bmatrix}$, then
	\begin{equation}
	\label{eq: disc_selfadj_E}
		\begin{alignedat}{2}
			(\bv{u}, E \bv{v})_{\bar{H}} &= - c \sum_{k = w,e} ( e_k^{+} \tilde{\bv{u}}^{+} , e_k^{+} \tilde{\bv{v}}^{+} )_{H_k^{+}} + ( e_k^{-} \tilde{\bv{u}}^{-} , e_k^{-} \tilde{\bv{v}}^{-} )_{H_k^{-}}.
		\end{alignedat}
	\end{equation}
	Since all the terms on the right-hand side of \eqref{eq: disc_selfadj_E} are symmetric, i.e. we can swap $\bv{u}$ and $\bv{v}$ and obtain the same expression, we have
	\begin{equation}
		(\bv{u}, E \bv{v})_{\bar{H}} = (\bv{v}, E \bv{u})_{\bar{H}} = (E \bv{u}, \bv{v})_{\bar{H}},
	\end{equation}
	which proves the lemma.
\end{proof}


\section{The discrete optimization problem}
\label{sec: the_opt_prob}
We now return to the optimization problem \eqref{eq: cont_min_prob_multiblock}. As discussed in Section \ref{sec: intro}, when analyzing PDE-constrained optimization problems of this kind one typically has two choices. Either you derive the adjoint equations and the associated gradient of the loss functional in the continuous setting and then discretize (OD). Or, you discretize the forward problem before deriving the discrete adjoint equations and gradient (DO). Here we do a compromise and discretize in space, but leave time continuous, and then proceed with the optimization. As we shall see, in the semi-discrete setting the spatial discretization guarantees that DO and OD are equivalent, due to Lemmas \ref{lemma: D_self_adj} and \ref{lemma: E_self_adj}. With space discretized, we have the ODE-constrained optimization problem
\begin{subequations}
\label{eq: disc_min_prob}
\begin{equation}
	\label{eq: disc_loss_func}
		\min_{\bv{p}} {\mathcal{J}(\bv{w},\bv{p}) = \frac{1}{2} \int _0^T} r(\bv{w},t)^2 \: dt + \frac{1}{2} \gamma \snorm{D_2 \bv{p}}_H, \quad \text{such that} \\
\end{equation}
\begin{equation}
	\label{eq: disc_wave_eq}
	\begin{alignedat}{2}
		&\bv{w}_{tt} = D \bv{w} + E \bv{w}_t + f(t) \bv{d}_s, \quad &&t > 0\\
		&\bv{w} = \bv{w}_t = 0, &&t = 0,
	\end{alignedat}
\end{equation}
\end{subequations}
where $r(\bv{w},t)$ is the residual of the loss given by
\begin{equation}
	r(\bv{w},t) = (\bv{d}_r, \bv{w}(t))_{\bar{H}} - \bv{v}_d(t),
\end{equation}
where $\bv{v}_d(t)$ is the target data in the receiver and $\bv{d}_r$ is constructed analogously to $\bv{d}_s$. Note that the term $(\bv{d}_r, \bv{w}(t))_{\bar{H}}$ is an interpolation of $\bv{w}$ onto the receiver coordinates $\bv{x}_r$.
\subsection{Shape parameterization and regularization}
\label{subsec: shape_par_regul}
The shape of the bottom boundary can be parametrized in many ways, see \cite{samareh2001survey} for an overview of common methods. Here we let $\bv{p} = [p_1,p_2,p_3,\dots,p_{m_{\Gamma_I}}]$ be a vector containing the $y$-coordinates of the bottom for each grid point in the $x$-direction, see Figure \ref{fig: lake}. To combat the ill-posedness of \eqref{eq: cont_min_prob_multiblock} a regularization term $\frac{1}{2} \gamma \snorm{D_2 \bv{p}}_H$ is added to \eqref{eq: disc_loss_func}. This additional term penalizes spurious oscillations in $\bv{p}$ which helps us find smooth shapes (i.e. reduces the risk of getting stuck in a non-regular local minima to \eqref{eq: disc_min_prob}). The parameter $\gamma$ is chosen experimentally, a small value will give rise to oscillations and lead to a non-regular solution, whereas a large value will restrict $\bv{p}$ and lead to suboptimal solutions.
 
{Using linear transfinite interpolation, the elements in discrete column-major ordered coordinate vectors
$\bv{x}$, and $\bv{y}$ in $\Omega_p^-$ are given by
\begin{equation}
	\label{eq: disc_xy_coords}
           x_{i,j} = x_l + (x_r-x_l)\xi_i \quad \text{and} \quad y_{i,j} = p_i + (L_I - p_i) \eta_j,
\end{equation}
where $\xi_i$, $\eta_j$ are the elements in the 1D coordinate vectors in $\tilde{\Omega}_p^-$, and $L_I$ is the $y$-coordinate of the interface, here $L_I = 0.5$.} 

\subsection{The adjoint equation method}
To solve \eqref{eq: disc_min_prob} efficiently gradient-based optimization will be employed, and to this end, we require $\pder{\mathcal{J}}{p_i}, i = 1,2,...,m_{\Gamma_I}$, given by
\begin{equation}
	\label{eq: disc_naive_grad}
	\begin{alignedat}{2}
		\pder{\mathcal{J}}{p_i} &= \frac{1}{2} \int _0^T \pder{}{p_i} r(\bv{w},t)^2 \: dt + \gamma (D_2 \bv{e}_i,D_2 \bv{p})_H = \\
		&= \int _0^T  (r(\bv{w},t) \bv{d}_r , \pder{\bv{w}}{p_i})_{\bar{H}} \: dt + \gamma (D_2 \bv{e}_i,D_2 \bv{p})_H,
	\end{alignedat}
\end{equation}
where $\bv{e}_i$ is a column vector with the entry 1 at position $i$ and zeros elsewhere. Here we have used that $\pder{\bv{d}_r}{p_i} = 0$ (since the receiver is located in $\Omega^+)$. Clearly, naively computing \eqref{eq: disc_naive_grad} requires evaluating $\pder{\bv{w}}{p_i}$. Approximating $\pder{\bv{w}}{p_i}$ using first-order finite differences, e.g., would require $m_{\Gamma_I}+1$ solves of \eqref{eq:  disc_wave_eq} which quickly becomes costly for large problem sizes. Instead, we introduce a Lagrange multiplier $\bm{\lambda} \in \mathbb{R}^N$ and use the adjoint framework \cite{Giannakoglou2008,plessix2006review}, which allows us to compute $\pder{\mathcal{J}}{p_i}$ without evaluating or approximating $\pder{\bv{w}}{p_i}$. The method is summarized in the following Lemma:
\begin{lemma}
	\label{lemma: grad_formula}
	Let $\bv{w}$ be a solution to \eqref{eq: disc_wave_eq} and $\bm{\lambda}$ a solution to the adjoint equation
	\begin{equation}
		\label{eq: disc_adj_ode}
		\begin{alignedat}{2}
			&\bm{\lambda}_{\tau \tau} = D \bm{\lambda} + E \bm{\lambda}_\tau - r(\bv{w},\tau) \bv{d}_r, \quad 0 \leq \tau \leq T, \\
			&\bm{\lambda} = \bm{\lambda}_\tau = 0, \quad \tau = 0,
		\end{alignedat}
\end{equation}
where $\tau = T - t$. Then the gradient is given by
\begin{equation}
	\label{eq: disc_grad_formula}
	\pder{\mathcal{J}}{p_i} = \int_0^T -(\bm{\lambda},\pder{D}{p_i} \bv{w})_{\bar{H}} + (\bm{\lambda}_t,\pder{E}{p_i} \bv{w})_{\bar{H}} \: dt + \gamma (D_2 \bv{e}_i,D_2 \bv{p})_H.
\end{equation}
\end{lemma}
\begin{proof}
First, we define the Lagrangian functional
\begin{equation}
	\mathcal{L}(\bv{w}, \bm{\lambda}) = \mathcal{J} + \int _0^T (\bm{\lambda},\bv{w}_{tt} - D \bv{w} - E \bv{w}_t - f(t) \bv{d}_s)_{\bar{H}} \: dt,
\end{equation}
and note that $\mathcal{J} = \mathcal{L}$ and thus $\pder{\mathcal{J}}{p_i} = \pder{\mathcal{L}}{p_i}$ for any $\bm{\lambda}$ whenever $\bv{w}$ is a solution to \eqref{eq: disc_wave_eq}. Consider the gradient of $\mathcal{L}$, given by
\begin{equation}
	\label{eq: disc_opt_grad_1}
	\pder{\mathcal{L}}{p_i} = \pder{\mathcal{J}}{p_i} + \int _0^T (\bm{\lambda},\pder{\bv{w}_{tt}}{p_i} - \pder{D \bv{w}}{p_i} - \pder{E \bv{w}_t}{p_i})_{\bar{H}} \: dt,
\end{equation}
where we have used that $\pder{\bv{d}_s}{p_i} = 0$ (since the source is located in $\Omega^+)$. The first term in \eqref{eq: disc_opt_grad_1} is given by \eqref{eq: disc_naive_grad}. The other terms are treated using integration by parts in time, resulting in
\begin{equation}
	\label{eq: disc_opt_grad_2}
	\begin{alignedat}{2}
		\pder{\mathcal{L}}{p_i} &= \pder{\mathcal{J}}{p_i} + \int _0^T (\bm{\lambda}_{tt},\pder{\bv{w}}{p_i})_{\bar{H}} - (\bm{\lambda},\pder{D \bv{w}}{p_i})_{\bar{H}} + (\bm{\lambda}_{t},\pder{E \bv{w}}{p_i})_{\bar{H}} \: dt \\
		&+ \left [ (\bm{\lambda},\pder{\bv{w}_t}{p_i})_{\bar{H}} - (\bm{\lambda}_t,\pder{\bv{w}}{p_i})_{\bar{H}} - (\bm{\lambda},\pder{E \bv{w}}{p_i})_{\bar{H}} \right ]^T_0 \\
		&= \pder{\mathcal{J}}{p_i} + \int _0^T (\bm{\lambda}_{tt},\pder{\bv{w}}{p_i})_{\bar{H}} - (\bm{\lambda},\pder{D \bv{w}}{p_i})_{\bar{H}} + (\bm{\lambda}_{t},\pder{E \bv{w}}{p_i})_{\bar{H}} \: dt,
	\end{alignedat}
\end{equation}
where we have used the initial conditions for $\bv{w}$ and $\bv{w}_t$ and prescribed the following terminal conditions for $\bm{\lambda}$ and $\bm{\lambda}_t$:
\begin{equation}
	\bm{\lambda} = 0, \quad \bm{\lambda}_t = 0, \quad t = T.
\end{equation}
Using \eqref{eq: disc_naive_grad} and Lemmas \ref{lemma: D_self_adj} and \ref{lemma: E_self_adj} we get
\begin{equation}
	\begin{alignedat}{2}
		\pder{\mathcal{L}}{p_i} &= \int _0^T  (r(\bv{w},t) \bv{d}_r + \bm{\lambda}_{tt} - D \bm{\lambda} + E \bm{\lambda}_t , \pder{\bv{w}}{p_i})_{\bar{H}} \: dt \\
		&+ \int_0^T -(\bm{\lambda},\pder{D}{p_i} \bv{w})_{\bar{H}} + (\bm{\lambda}_t,\pder{E}{p_i} \bv{w})_{\bar{H}} \: dt + \gamma (D_2 e_i,D_2 \bv{p})_{\bar{H}}.
	\end{alignedat}
\end{equation}
If $\bm{\lambda}$ satisfies \eqref{eq: disc_adj_ode} we get the following formula for the gradient:
\begin{equation}
	\pder{\mathcal{L}}{p_i} = \int_0^T -(\bm{\lambda},\pder{D}{p_i} \bv{w})_{\bar{H}} + (\bm{\lambda}_t,\pder{E}{p_i} \bv{w})_{\bar{H}} \: dt + \gamma (D_2 \bv{e}_i,D_2 \bv{p})_H,
\end{equation}
and since $\pder{\mathcal{L}}{p_i} = \pder{\mathcal{J}}{p_i}$ we get \eqref{eq: disc_grad_formula}.
\end{proof}
Note that the only difference between \eqref{eq: disc_adj_ode} and \eqref{eq: disc_wave_eq} is the forcing function, and hence stability of \eqref{eq: disc_adj_ode} follows immediately from Lemma \ref{lemma: ode_stability}. The matrices $\pder{D}{p_i}$ and $\pder{E}{p_i}$ can be computed analytically and are presented in \ref{sec: op_dervs}. Essentially, the matrices consist of the SBP operators together with derivatives of the metric coefficients given in Section \ref{sec: discr_omega_minus}. The derivation involves repeated application of the product rule but is otherwise straightforward.

\begin{remark}
	In this model problem, we have assumed that the source and receiver are located in $\Omega^+$, which simplifies the analysis since $\pder{\bv{d}_s}{p_i} = \pder{\bv{d}_r}{p_i} = 0$ holds. If this is not the case, one would have to evaluate the derivative of the discrete Dirac delta function, which is not well-defined everywhere for the discrete Dirac delta functions used here. In \cite{Sjogreen2014}, point source discretizations that are continuously differentiable everywhere are derived. In principle, we could use this discretization instead and allow $\bv{x}_s$ and $\bv{x}_r$ to be located in $\Omega_p^-$, but this is out of scope for the present work.
\end{remark}

\subsection{Dual consistency}
In the continuous setting it is well-established that \eqref{eq: cont_wave_eq_multiblock} is self-adjoint under time reversal \cite{gauthier_1986,plessix2006review,Almquist2020,doi:10.1190/geo2022-0195.1} such that for the continuous adjoint state variables $\lambda^\pm$ the adjoint (or dual) equations are given by
\begin{equation}
	\label{eq: cont_adj_wave_eq_multiblock}
	\begin{array}{lll}
		\lambda^+_{\tau\tau} = c^2\Delta \lambda^+ - r(u,\tau) \hat{\delta}(\bv{x} - \bv{x}_r), 	&\bv{x} \in \Omega^+, 		&\tau \in [0,T], \\
		\lambda^-_{\tau\tau} = c^2\Delta \lambda^-, 									&\bv{x} \in \Omega^-_p, 	&\tau \in [0,T], \\
		\lambda^\pm_{\tau} + c\bv{n}^\pm \cdot \nabla \lambda^\pm = 0, \quad &\bv{x} \in \del \Omega_p^{(\pm,w,e)},  &\tau \in [0,T], \\
		\bv{n}^- \cdot \nabla \lambda^- = 0, \quad &\bv{x} \in \del \Omega_p^{(-,s)}, \quad &\tau \in [0,T], \\
		\lambda^+ = 0, \quad &\bv{x} \in \del \Omega_p^{(+,n)}, \quad &\tau \in [0,T], \\
		\lambda^+ - \lambda^- = 0,												& \bv{x} \in \Gamma_I, &\tau \in [0,T], \\
		n^+ \cdot \nabla \lambda^+ + n^- \cdot \nabla \lambda^- = 0,			& \bv{x} \in \Gamma_I, &\tau \in [0,T], \\
		\lambda^\pm = 0,  \quad \lambda^\pm_\tau = 0, 										&\bv{x} \in \Omega, 		&\tau = 0. \\
	\end{array}
\end{equation}
Note that the semi-discrete adjoint problem \eqref{eq: disc_adj_ode} is a consistent approximation of \eqref{eq: cont_adj_wave_eq_multiblock}, and thus \eqref{eq: disc_wave_eq} is a dual consistent semi-discretization of the forward problem \cite{berg_nordstrom_2012,HICKEN2014161}. Moreover, the gradient to the continuous optimization problem \eqref{eq: cont_min_prob_multiblock} is given by the following lemma:
\begin{lemma}\label{lemma: cont_grad}
	Let $\lambda^-$ satisfy \eqref{eq: cont_adj_wave_eq_multiblock}. Then the gradient $\frac{\delta \mathtt{J}}{\delta p}$ to \eqref{eq: cont_min_prob_multiblock}, is given by
	\begin{equation}\label{eq: cont_grad}
		\begin{aligned}
		\frac{\delta \mathtt{J}}{\delta p} = \int_{0}^{T} G_{\lambda} + G_{\lambda_t} dt,
		\end{aligned}
	\end{equation}
	where
	\begin{equation}\label{eq: cont_grad_lambda}
		\begin{aligned}
		G_{\lambda} = -c^2\Big( -&(\lambda^-, J^{-1}\frac{\delta J}{\delta p}\Delta u^-)_{\Omega^- }  \\
								+&(\lambda^-,J^{-1}((\frac{\delta \alpha_1}{\delta p} u^-_\xi + \frac{\delta \beta}{\delta p} u^-_\eta)_\xi + (\frac{\delta \alpha_2}{\delta p} u^-_\eta + \frac{\delta \beta}{\delta p} u^-_\xi)_\eta))_{\Omega^- } \\
		-&\langle \lambda^-, W_1^{-1}(\frac{\delta \alpha_1}{\delta p}u^-_\xi + \frac{\delta\beta}{\delta p} u^-_\eta )\rangle_{\partial\Omega^{(-,e)} } \\
		+&\langle \lambda^-, W_1^{-1}(\frac{\delta \alpha_1}{\delta p}u^-_\xi + \frac{\delta\beta}{\delta p} u^-_\eta )\rangle_{\partial\Omega^{(-,w)} } \\
		-&\langle\lambda^-, W_2^{-1}(\frac{\delta\alpha_2}{\delta p}u^-_\eta + \frac{\delta\beta}{\delta p} u^-_\xi )\rangle_{\partial\Omega^{(-,n)}}\\
		+&\langle\lambda^-, W_2^{-1}(\frac{\delta\alpha_2}{\delta p}u^-_\eta + \frac{\delta\beta}{\delta p} u^-_\xi )\rangle_{\partial\Omega^{(-,s)} } \Big),
		\end{aligned}
	\end{equation}
	and
	\begin{equation}\label{eq: cont_grad_lambda_t}
		\begin{aligned}
			G_{\lambda_t} = &c\Big(\langle \lambda^-_t, W_1^{-1}\frac{\delta W_1}{\delta p} u^- \rangle_{\partial \Omega^{(-,e)}} + \langle \lambda^-_t, W_1^{-1}\frac{\delta W_1}{\delta p} u^- \rangle_{\partial \Omega^{(-,w)}}\Big),
		\end{aligned}
	\end{equation}
	with the metric coefficients $J$, $\alpha_1$, $\alpha_2$, $\beta$, $W_1$ and $W_1$, given in Section \ref{sec: cont_analysis}.
\end{lemma}
\begin{proof}
	See \ref{sec: cont_grad_proof}.
\end{proof}
Lemma \ref{lemma: cont_grad} is the continuous counterpart of Lemma \ref{lemma: grad_formula}, where \eqref{eq: disc_grad_formula} is the semi-discrete version of \eqref{eq: cont_grad}, with the addition of regularization. This is seen by comparing the explicit expressions for the operator derivatives presented in \ref{sec: op_dervs} to \eqref{eq: cont_grad_lambda} - \eqref{eq: cont_grad_lambda_t} where $G_{\lambda} \approx (\bm{\lambda}, \frac{\partial D}{\partial p_i} \bv{w})_{\bar{H}}$ and $G_{\lambda_t} \approx (\bm{\lambda}_t,\frac{\partial E}{\partial p_i} \bv{w})_{\bar{H}}$.

\subsection{Temporal discretization}
For time discretization we use the 4th-order explicit Runge--Kutta method (RK4) and for computing the time integrals a 6th-order accurate SBP quadrature. {The time step is chosen as
\begin{equation}
\Delta t = k \Delta t_{max},
\end{equation}
where 
\begin{equation}
\Delta t_{max} = \frac{2.8}{\sqrt{\rho(D)}},
\end{equation}
is approximately the stability limit of RK4 applied to \eqref{eq: disc_ODE_system} (written on first-order form), $\rho(D)$ is the spectral radius of $D$, and $k < 1$ is a positive CFL constant. The influence of $E$ on the stability limit $\Delta t_{max}$ is in this case very small since it scales as $h^{-1}$, where $h$ is the spatial step size, compared to $D$ that scales as $h^{-2}$. 

Since the temporal discretization used here is not self-adjoint, the computed gradients will contain an approximation error. In Section \ref{subsec: bath} we verify the accuracy of the gradient and show that it decreases with $k$. For the numerical experiments presented in this paper, we use $k = 0.1$, which seems to be sufficient for the optimization algorithm to work well on the problems we solve. There exist self-adjoint temporal discretization schemes that would lead to exact gradients, e.g., symplectic Runge--Kutta methods \cite{sanz_serna} or SBP in time \cite{nordstrom_lundquist_2016} where the temporal derivatives are also approximated using SBP operators, but implementing these are out of scope in the present work.
}
\subsection{Optimization algorithm}
There are many optimization algorithms that could be employed for these types of problems. Here we use the BFGS algorithm \cite{10.1093/imamat/6.1.76} as implemented in the Matlab function \emph{fminunc}, which is a quasi-newton method that approximates the Hessian using only the gradient of the loss function. For larger problems (e.g., with higher grid resolutions or 3D problems) the L-BFGS method would be a suitable alternative. Each iteration of the BFGS method requires the loss $\mathcal{J}$ and the gradient $\nabla_\bv{p} \mathcal{J}$, which are computed as follows:

\begin{enumerate}
	\item Solve the forward problem \eqref{eq: disc_wave_eq}.
	\item Compute the loss \eqref{eq: disc_loss_func} using the solution to the forward problem.
	\item Solve the adjoint problem \eqref{eq: disc_adj_ode} using the solution to the forward problem.
	\item For each $i = 1,2,...,m_{\Gamma_I}$, compute $\pder{\mathcal{J}}{p_i}$ using \eqref{eq: disc_grad_formula} and form the gradient vector
	\begin{equation}
		\nabla_\bv{p} \mathcal{J} = \left [\pder{\mathcal{J}}{p_1},\pder{\mathcal{J}}{p_2},...,\pder{\mathcal{J}}{p_{m_{\Gamma_I}}} \right ].
	\end{equation}
\end{enumerate}
Note that we only have to solve the forward and adjoint problems one time each per iteration, independently of the number of optimization parameters $m_{\Gamma_I}$. This is one of the main advantages of the adjoint method. {However, evaluating \eqref{eq: disc_grad_formula} in step 4 requires storing the vectors $\bv{w}$, $\bm{\lambda}$, and $\bm{\lambda}_t$ for all time steps, which naturally may be very memory consuming, especially for long-time simulations in realistic 3D applications. In such cases, one often must resort to check-pointing (storing the vectors at certain time intervals and recomputing intermediate results) or storing the vectors on disk.}

\section{Numerical experiments}
\label{sec: num_exp}
\subsection{Accuracy study}
Since the SBP-P-SAT method is new in this setting (wave equations for multiblock, curvilinear domains), we briefly verify the accuracy of the method by performing a convergence study on the forward problem with a known analytical solution. We consider a circular domain $\Omega$ decomposed into five blocks as depicted in Figure \ref{fig: circ}, and solve the PDE
\begin{equation}
	\label{eq: cont_accuracy_pde}
	\begin{alignedat}{4}
		u_{tt} &=  \Delta u,  &&\bv{x} \in \Omega,  &&t \in [0,T], \\
		u &= g(\bv{x},t), \quad&&\bv{x} \in \del \Omega, \quad&&t \in [0,T], \\
		u &= u_0(\bv{x}),  &&\bv{x} \in \Omega, &&t = 0, \\
		u_t &= 0,  &&\bv{x} \in \Omega, &&t = 0, \\
	\end{alignedat}
\end{equation}
where $g(\bv{x},t)$ and $u_0(\bv{x})$ are chosen so that $u$ satisfies the standing-wave solution
\begin{equation}
	u(\bv{x},t) = \sin(3 \pi x) \sin(4 \pi y) \cos(5 \pi t).
\end{equation}
\begin{figure}[!htbp]
	\centering
	\includegraphics[width=0.5\textwidth]{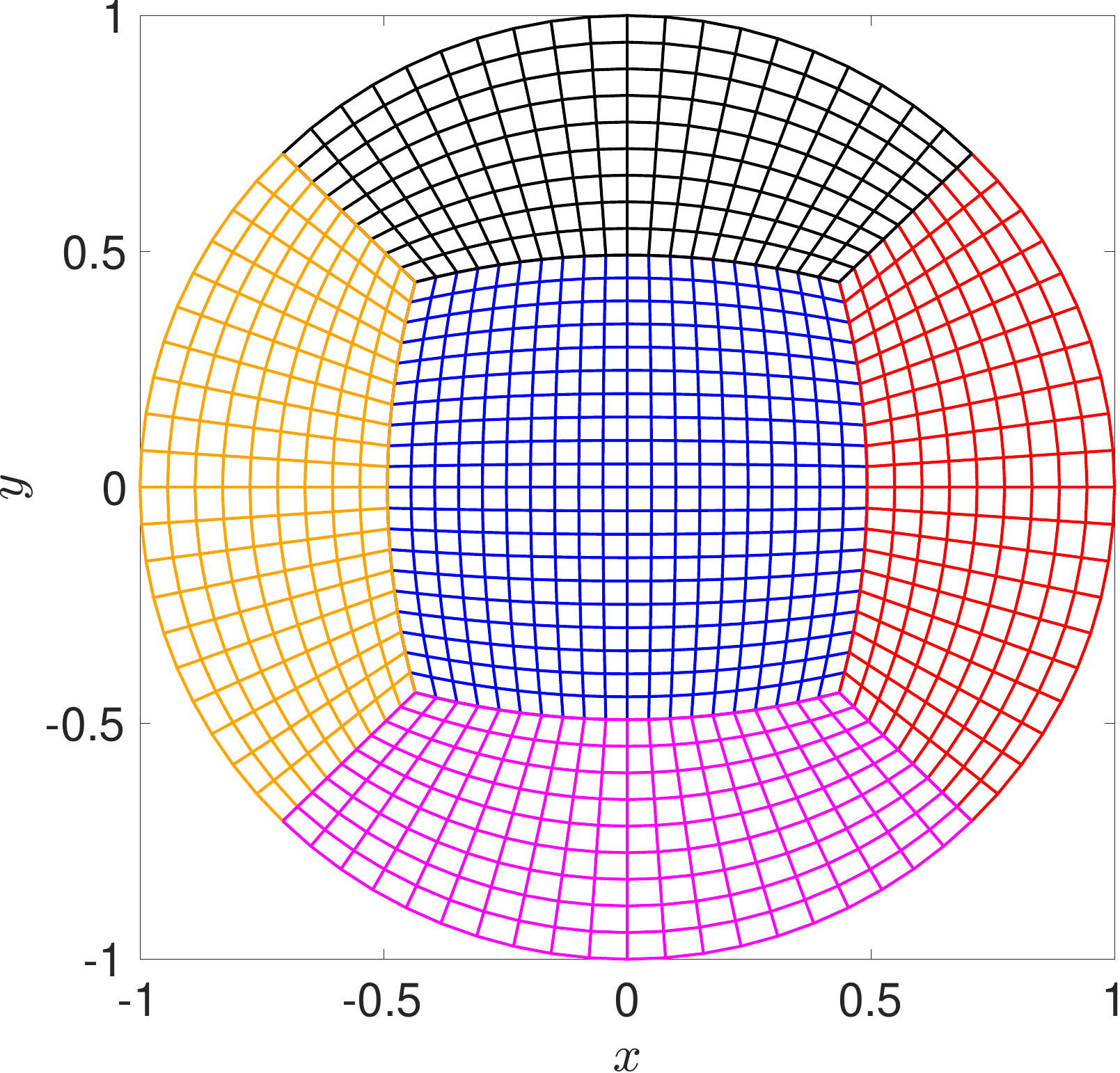}
	\caption{Grid of the circular domain. The different colors of the grid indicate the block decomposition.}
	\label{fig: circ}
\end{figure}

The problem \eqref{eq: cont_accuracy_pde} is discretized using the SBP-P-SAT method described in Section \ref{sec: forward_problem}, with interface conditions imposed using the hybrid projection and SAT method and the Dirichlet boundary conditions imposed using the projection method. 

In Table \ref{tabl: errconv} the $L_2$-error and convergence rate at final time $T = 1$ is presented for the 4th- and 6th-order accurate SBP operators, where each row corresponds to a different number of total grid points $N$. With the 4th-order accurate operators we obtain the convergence rate 4 while with the 6th-order accurate operators the convergence rate is slightly above 5, which is in line with previous observations \cite{Mattsson2006}.
\begin{table}
	\centering
	\caption{$L_2$-error and convergence rate with varying total degrees of freedom $N$ for standing-wave problem with 4th- and 6th-order accurate SBP operators.}
	\label{tabl: errconv}
	\begin{tabular}{|c||c|c||c|c|} 
	\hline
	$N$ & $\log_{10}(e_4)$ & $q_4$ & $\log_{10}(e_6)$ & $q_6$\\ 
	\hline
	4797 &	-2.98 &	- & -3.25 &	- \\
	\hline
	10553 &	-3.71 &	-4.29 & -4.24 &	-5.77 \\
	\hline
	18549 &	-4.22 &	-4.18 & -4.94 &	-5.68 \\
	\hline
	28381	& -4.61 &	-4.21 & -5.47 &	-5.77 \\
	\hline
	40777 &	-4.93 &	-4.10  & -5.86	& -5.04 \\
	\hline
	72289 &	-5.44 &	-4.08 & -6.57 &	-5.68 \\
	\hline
	112761 & -5.83 & -4.05 & -7.08 &	-5.30 \\
	\hline
	\end{tabular}
\end{table}
\subsection{Bathymetry optimization} 
\label{subsec: bath}
We now consider the optimization problem \eqref{eq: cont_min_prob_multiblock}. The domain is given by Figure \ref{fig: lake} and the wave speed is $c = 1$. The source is located at $[0.25,0.8]$ with the time-dependent function given by the Ricker wavelet function
\begin{equation}
	f(t) = \frac{2}{\sqrt{3\sigma} \pi^{1/4}} \left ( 1 - \left ( \frac{t}{\sigma} \right)^2 \right) e^{-\frac{t^2}{2\sigma^2}},
\end{equation}
with $\sigma = 0.1$. The receiver is located at $[0.75,0.8]$. We use synthetic receiver data, produced by simulating the forward problem with a bottom boundary (seabed) given in Figure \ref{fig: lake}, using the 6th-order accurate operators with $m^+_\xi = m^-_x = 401$ and $m^+_\eta = m^-_y = 201$ (total degrees of freedom $N = 161202$). In the optimization, we use the 4th-order operators and $m^+_\xi = m^-_x = 41$ and $m^+_\eta = m^-_y = 21$ (total degrees of freedom $N = 1722$). Linear interpolation of the synthetic receiver data is used when there is no data matching the time level of the numerical solution. The final time is set to $T = 4$ and the regularization parameter to $\gamma = 10^{-5}$. The initial guess of the seabed is chosen as a straight line, i.e. $\bv{p} = [0,0,...,0]$.

\begin{figure}[!htbp]
	\centering
	\includegraphics[width=0.5\textwidth]{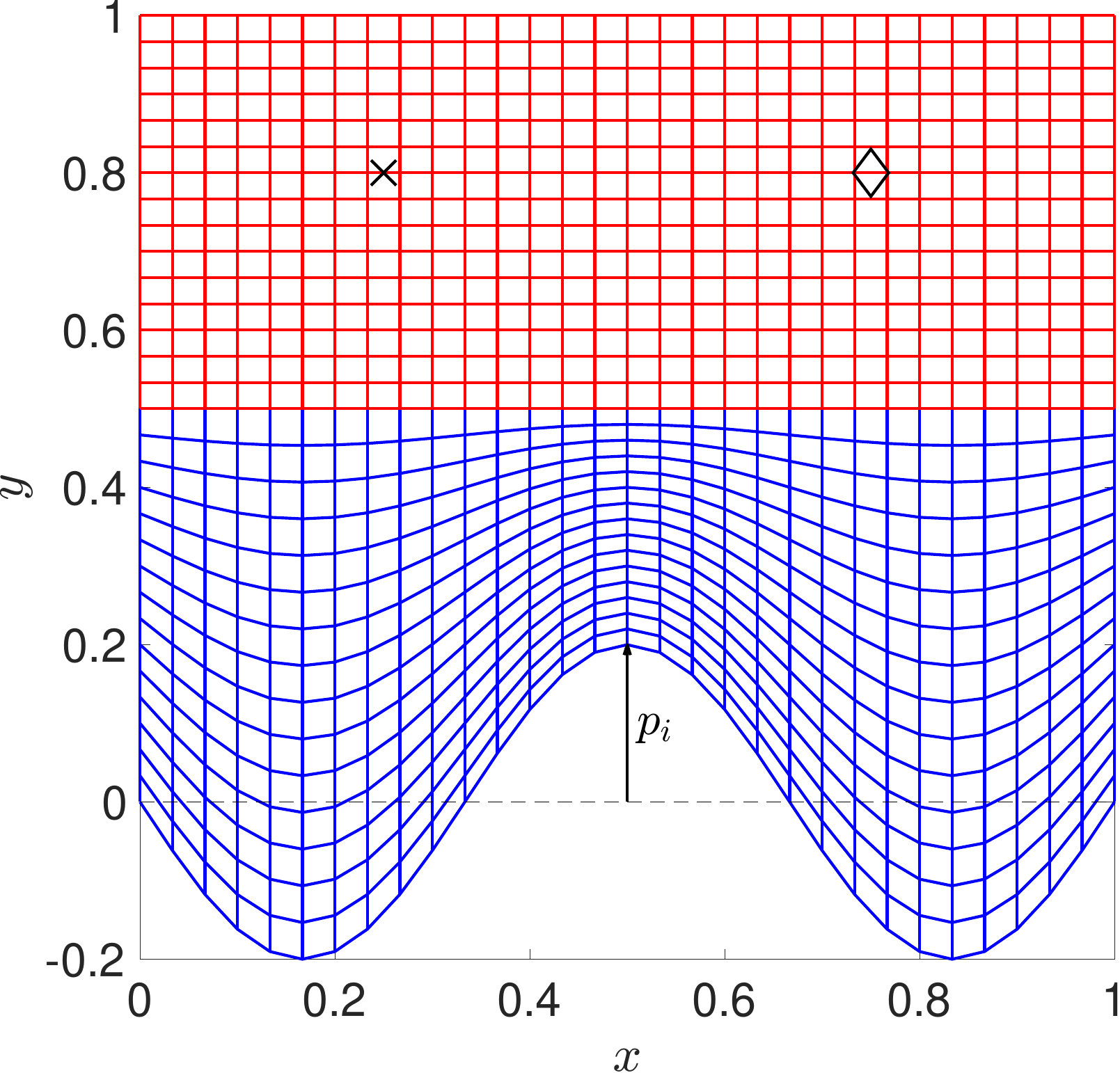}
	\caption{Discretization of bathymetry problem with parameterization of the seabed. The different colors of the grid indicate the block decomposition. The cross ($\bm{\times}$) and diamond ($\bm{\Diamond}$) indicate the location of the source and the receiver, respectively.}
	\label{fig: lake}
\end{figure}

{
We begin by investigating the accuracy of the gradient \eqref{eq: disc_grad_formula} by comparing it to the more common, but expensive, finite difference approach. Let 
\begin{equation}
 	(D^+ \mathcal{J}(\bv{p}))_i = \frac{\mathcal{J}(\bv{p} + \Delta p \bv{e}_i) - \mathcal{J}(\bv{p})}{\Delta p},
\end{equation} 
denote the $i$:th component of the gradient vector approximated using first-order finite differences, where $\bv{e}_i$ is the vector with a 1 at the $i$:th element and zeroes elsewhere and $\Delta p$ is the step size of finite difference approximation. Further, let
\begin{equation}
	e(\Delta p) = \frac{\norm{\nabla_\mathbf{p} \mathcal{J}(\bv{p})-D^+ \mathcal{J}(\bv{p})}_2}{\norm{\nabla_\mathbf{p} \mathcal{J}(\bv{p})}_2},
\end{equation}
denote the relative $L_2$-error of the finite difference gradient assuming that $\nabla_\mathbf{p} \mathcal{J}(\bv{p})$ is exact. In Figure \ref{fig: graderr}, the error $e(\Delta p)$ is plotted against $\Delta p$ for various CFL constants $k$. The results show that the error initially decreases with a first-order rate, but for small enough $\Delta p$ it plateaus at a constant error that decreases for smaller time steps. Hence we have a discretization error in the gradient arising from the temporal discretization. For $k=0.1$, which is what we use in the optimization, the relative error in the gradient is approximately 1\%. For $k = 0.001$ and $k = 0.0001$ and very small $\Delta p$ we also start seeing the effects of cancellation errors in the finite difference gradient.  
\begin{figure}[!htbp]
	\centering
	\includegraphics[width=0.6\textwidth]{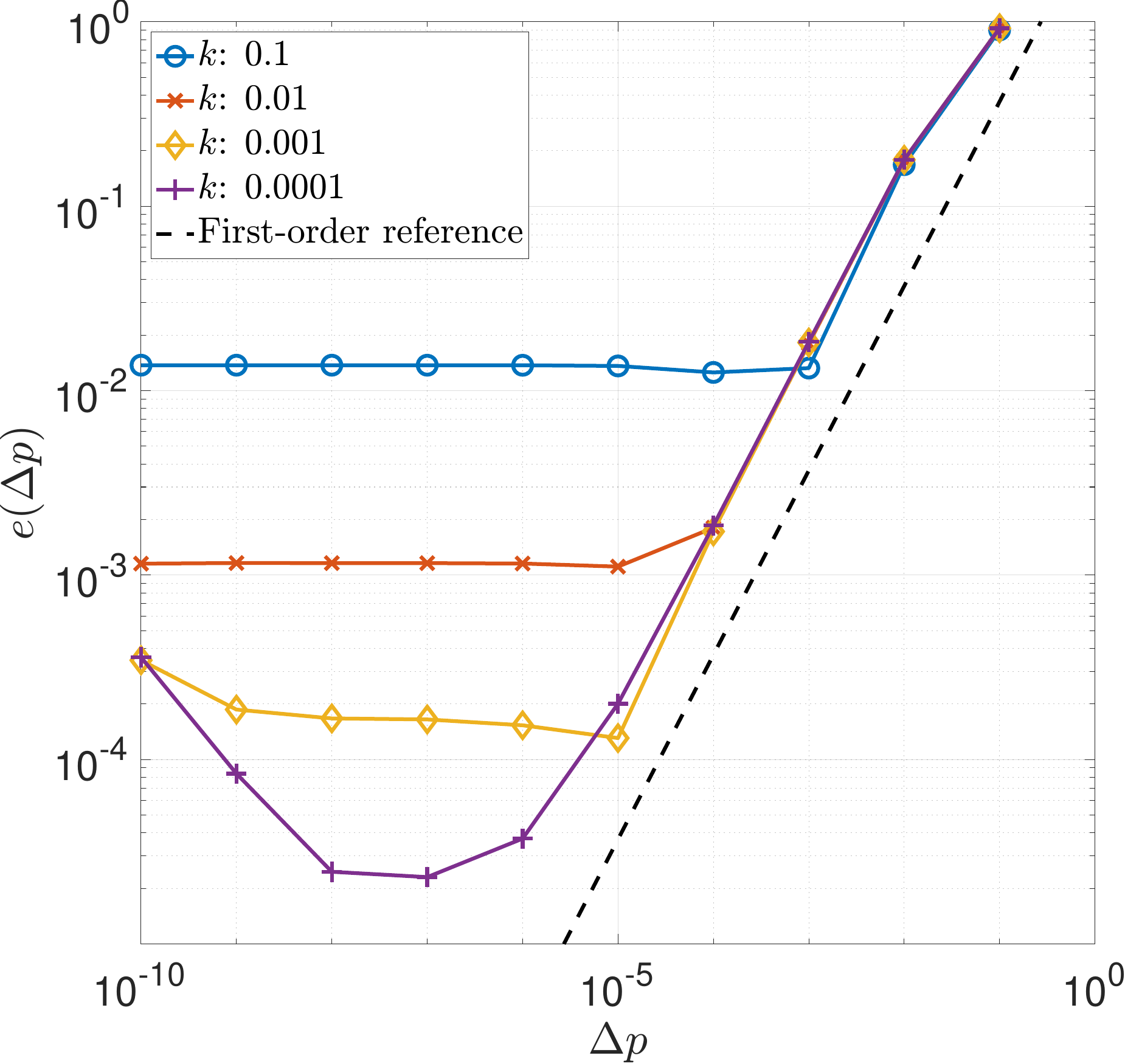}
	\caption{Relative error between gradient approximated using finite differences and gradient computed using the adjoint method for varying CFL constants $k$.}
	\label{fig: graderr}
\end{figure}

We now turn to the actual optimization. In Figure \ref{fig: lake_snaps} the shape of the seabed after 0, 5, 20, and 177 iterations are presented. After 177 iterations the optimization stops with the chosen tolerance $10^{-6}$. We can conclude that the method manages to reconstruct the shape of the seabed accurately with only one source and one receiver, and that a 1\% error in the gradient seems sufficiently small for this problem.}
\begin{figure}[!htbp]
	\centering
	\includegraphics[width=0.6\textwidth]{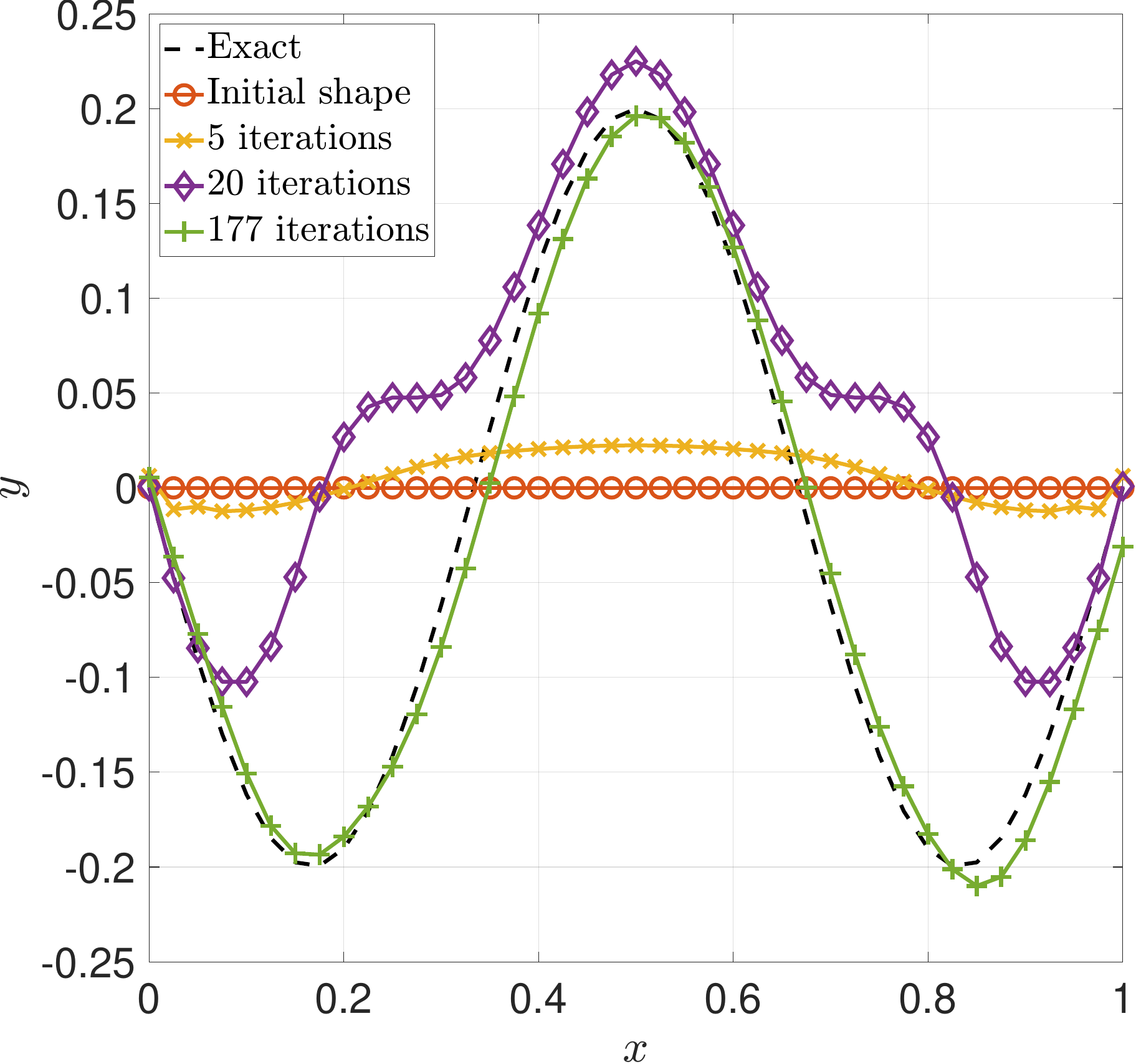}
	\caption{Shape of the seabed after 0, 5, 20, and 177 iterations.}
	\label{fig: lake_snaps}
\end{figure}
\subsection{Air horn shape optimization}
{For the final numerical example, we use a similar setup as in \cite{BANGTSSON20031533}, where the shape of the mouth of an acoustic horn is optimized to minimize the reflected sound, see Figure \ref{fig: horn_full}. The problem setup, including the domain decomposition and boundary conditions, is presented in Figure \ref{fig: horn_disc}. On the walls of the horn, $\Gamma_W$ and $\Gamma_{H,p}$, we prescribe fully reflecting homogeneous Neumann boundary conditions. At the outflow $\Gamma_O$ we use first-order accurate outflow boundary conditions. To reduce the computational costs by half, we utilize the horizontal symmetry and prescribe symmetry boundary conditions (homogeneous Neumann conditions) at $\Gamma_S$. Note that $\Gamma_{H,p}$ is parametrized by the optimization parameter $p$. At the inflow boundary $\Gamma_I$, a time-dependent inflow function is prescribed using inhomogeneous Dirichlet boundary conditions. The wave speed is $c = 340$ m/s and the final time $T = 0.05$ s.
\begin{figure}[!htbp]
	\centering
	\includegraphics[width=0.51\textwidth]{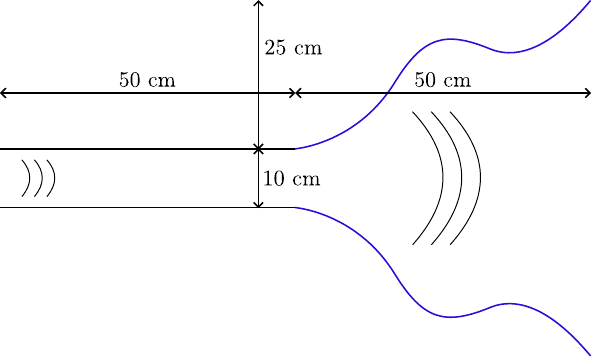}
	\caption{Dimensions of the acoustic horn. The blue section of the horn is subject to shape optimization.}
	\label{fig: horn_full}
\end{figure}

\begin{figure}[!htbp]
	\centering
    \begin{subfigure}[b]{0.9\textwidth}
         \centering
         \includegraphics[width=\textwidth]{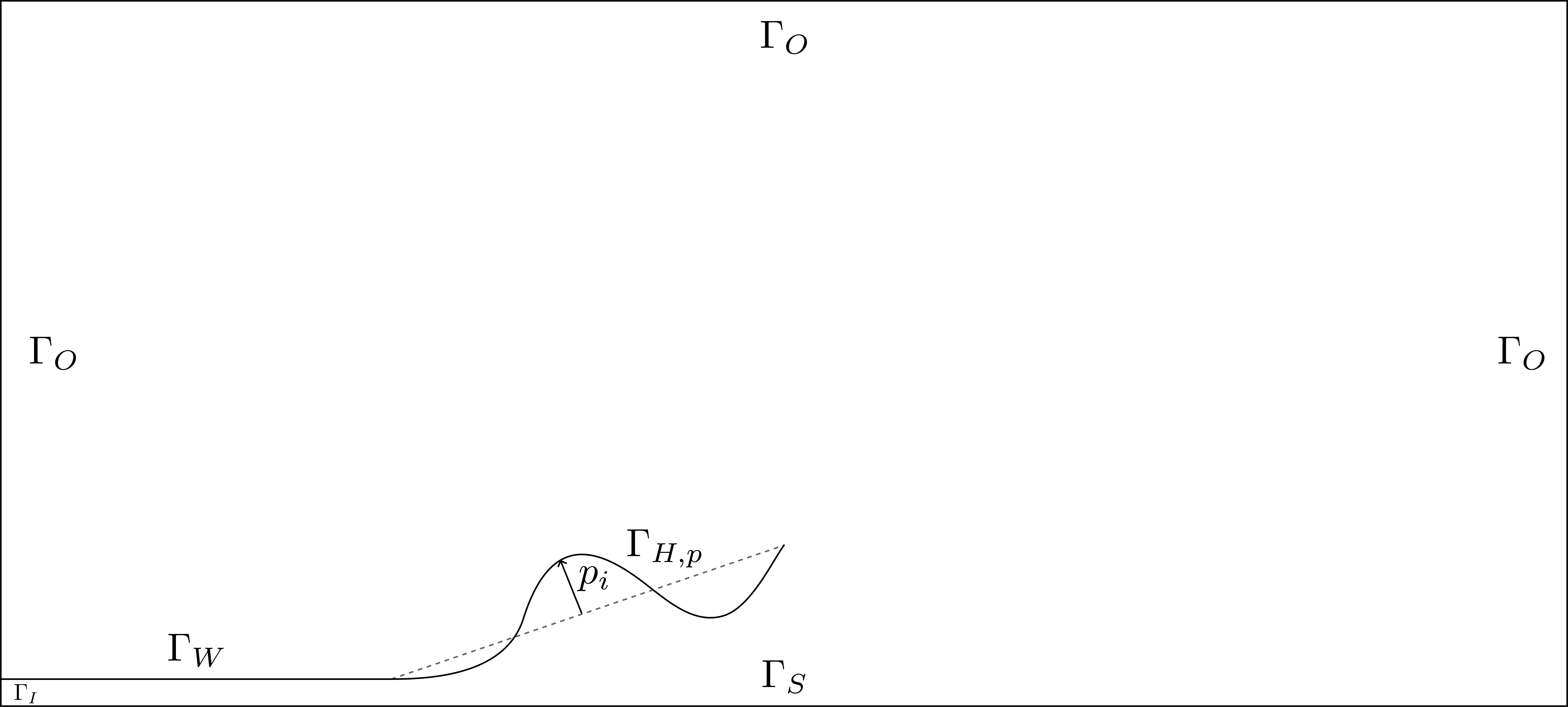}
         \caption{Boundary conditions and horn parameterization.}
         \label{fig: horn_disc_bc}
     \end{subfigure}
     \\ 
     \vspace{1cm}
     \begin{subfigure}[b]{0.965\textwidth}
     	\hspace{-0.49cm}
         \includegraphics[width=\textwidth]{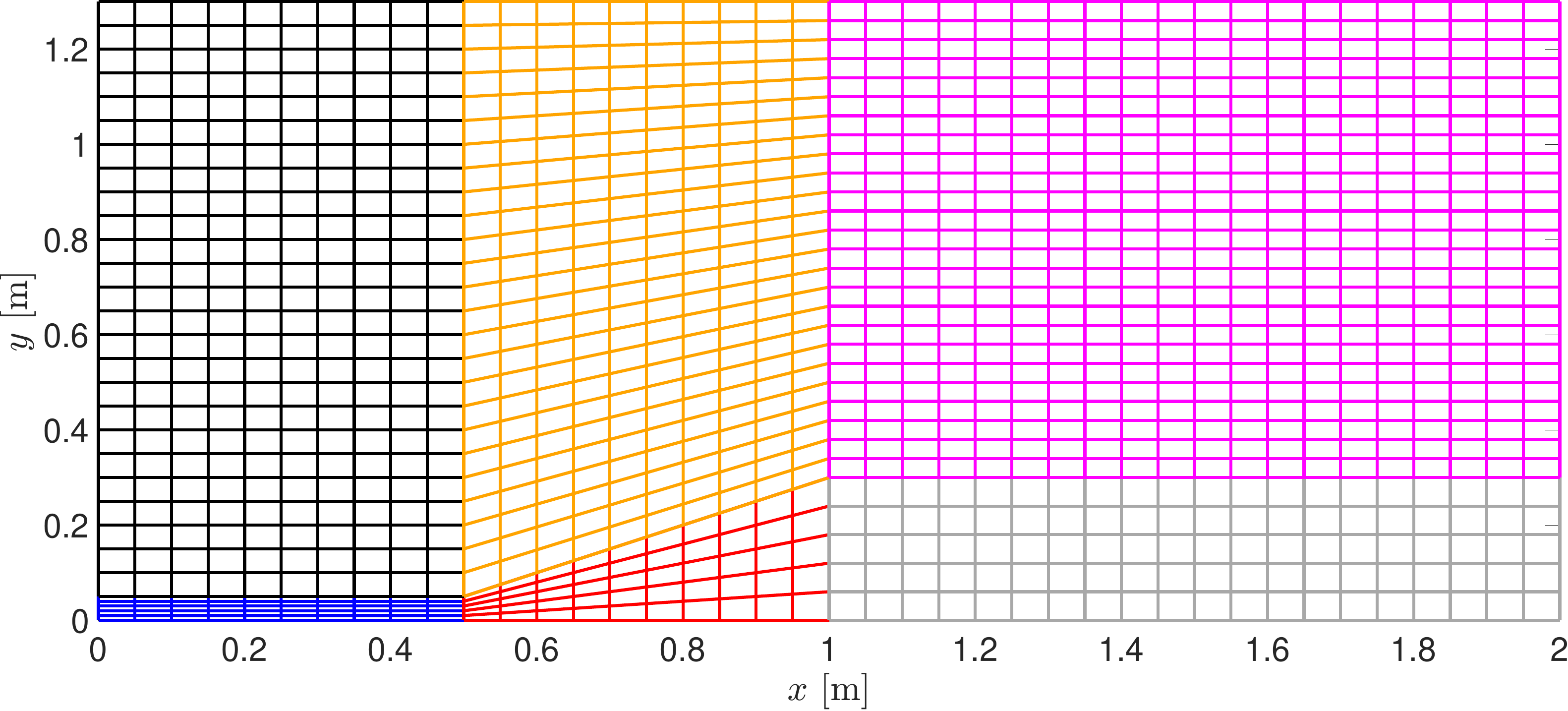}
         \caption{Coarse grid of the acoustic horn problem. The different colors indicate the block decomposition.}
         \label{fig: horn_disc_grid}
     \end{subfigure}
     \caption{Discretization of the acoustic horn problem.}
     \label{fig: horn_disc}
\end{figure}
The main difference between \cite{BANGTSSON20031533} and the present work is that we solve the acoustic wave equation in the time domain, whereas in \cite{BANGTSSON20031533} the optimization is performed in the frequency domain by solving Helmholtz equation. The advantage of our approach is that we can use any time-dependent inflow function we wish with very small effects on the computational costs, whereas in the frequency domain the Helmholtz equation must be solved one frequency at a time for complex signals. However, for simple inflow signals, the approach taken in \cite{BANGTSSON20031533} is likely to be the more computational efficient option. In this work we present results for two inflow functions at $\Gamma_I$, first a single-frequency wave and then a superposition of many sinusoidal waves with different frequencies.

Designing suitable objective functions for optimization problems is a non-trivial task. In \cite{BANGTSSON20031533}, a specific inflow boundary condition is derived that allows for straightforward computation of reflected waves, i.e. the left-going waves in the horn. In the optimization process, the mouth of the horn is modified to minimize the amplitudes of these waves. Here we take inspiration from \cite{BANGTSSON20031533} and design an objective function that also indicates left-going waves in the horn, but adapted for time-domain discretizations. This is achieved by minimizing the outgoing characteristic $u_t - c\bv{n} \cdot \nabla u$ integrated over the inflow boundary $\Gamma_I$ and time. Assuming that the flow in the channel from the inflow boundary to the mouth of the horn is predominately one-dimensional, the boundary integral $\|u_t - c\bv{n} \cdot \nabla u\|_{\Gamma_I}^2$ measures the outgoing energy flux at $\Gamma_I$. Therefore, minimizing $\int_0^T \|u_t - c\bv{n} \cdot \nabla u\|_{\Gamma_I}^2 dt$, minimizes the energy attributed to reflected waves in the horn. The unregularized optimization problem in the continuous setting then reads
\begin{subequations}
\label{eq: cont_min_prob_horn}
\begin{equation}
	\label{eq: cont_min_prob_horn_loss}
	\begin{alignedat}{2}
		\min_{p} \mathtt{J}(u,p) &= \frac{1}{2} \int _0^T \snorm{u_t - c\bv{n} \cdot \nabla u}_{\Gamma_I} \: dt, \\		
	\end{alignedat}	
\end{equation}
such that
\begin{equation}
	\label{eq: cont_wave_eq_horn}
	\begin{alignedat}{4}
		u_{tt} &=  c^2\Delta u, &&\bv{x} \in \Omega_p, &&t \in [0,T], \\
		\bv{n} \cdot \nabla u &= 0, \quad && \bv{x} \in \Gamma_W, \Gamma_{H,p}, \Gamma_S, \quad && t \in [0,T], \\
		u &= g(t), \quad && \bv{x} \in \Gamma_I, \quad && t \in [0,T], \\
		u_t + c\bv{n} \cdot \nabla u &= 0, \quad && \bv{x} \in \Gamma_O, \quad && t \in [0,T], \\
		u &= 0, \quad u_t = 0, \quad &&\bv{x} \in \Omega_p, &&t = 0,
	\end{alignedat}
\end{equation}
\end{subequations}
where $g(t)$ is the inflow function.

The shape of the mouth of the horn is parametrized by the deviation from a straight line between the two fixed points $(0.5,0.05)$ and $(1,0.3)$, see Figure \ref{fig: horn_disc_bc}. The problem \eqref{eq: cont_min_prob_horn} is discretized using the SBP-P-SAT method described in Section \ref{sec: forward_problem}, generalized to the multiblock setting in Figure \ref{fig: horn_disc}. As for the bathymetry problem, the adjoint problem to \eqref{eq: cont_wave_eq_horn} is derived which allows for efficient computation of the gradient of the loss function $\mathcal{J}$. To regularize the optimization problem we add an additional term to $\mathcal{J}$ as described in Section \ref{subsec: shape_par_regul}, with regularization parameter $\gamma = 10^{-7}$. To retain space we refrain from presenting the discretization in detail.

Two numerical experiments are performed to evaluate the method. First, we consider the problem with a single frequency $f = 300$ Hz, and inflow function
\begin{equation}
	g(t) = w(t) \sin(2 \pi f t),
\end{equation}
where $w(t)$ is a tapering window function shown in Figure \ref{fig: horn_window} \cite{McKechan_2010}. The window function is chosen so that at least five periods of the input signal with maximum amplitude are prescribed at the inflow.
\begin{figure}[!htbp]
	\centering
	\includegraphics[width=0.5\textwidth]{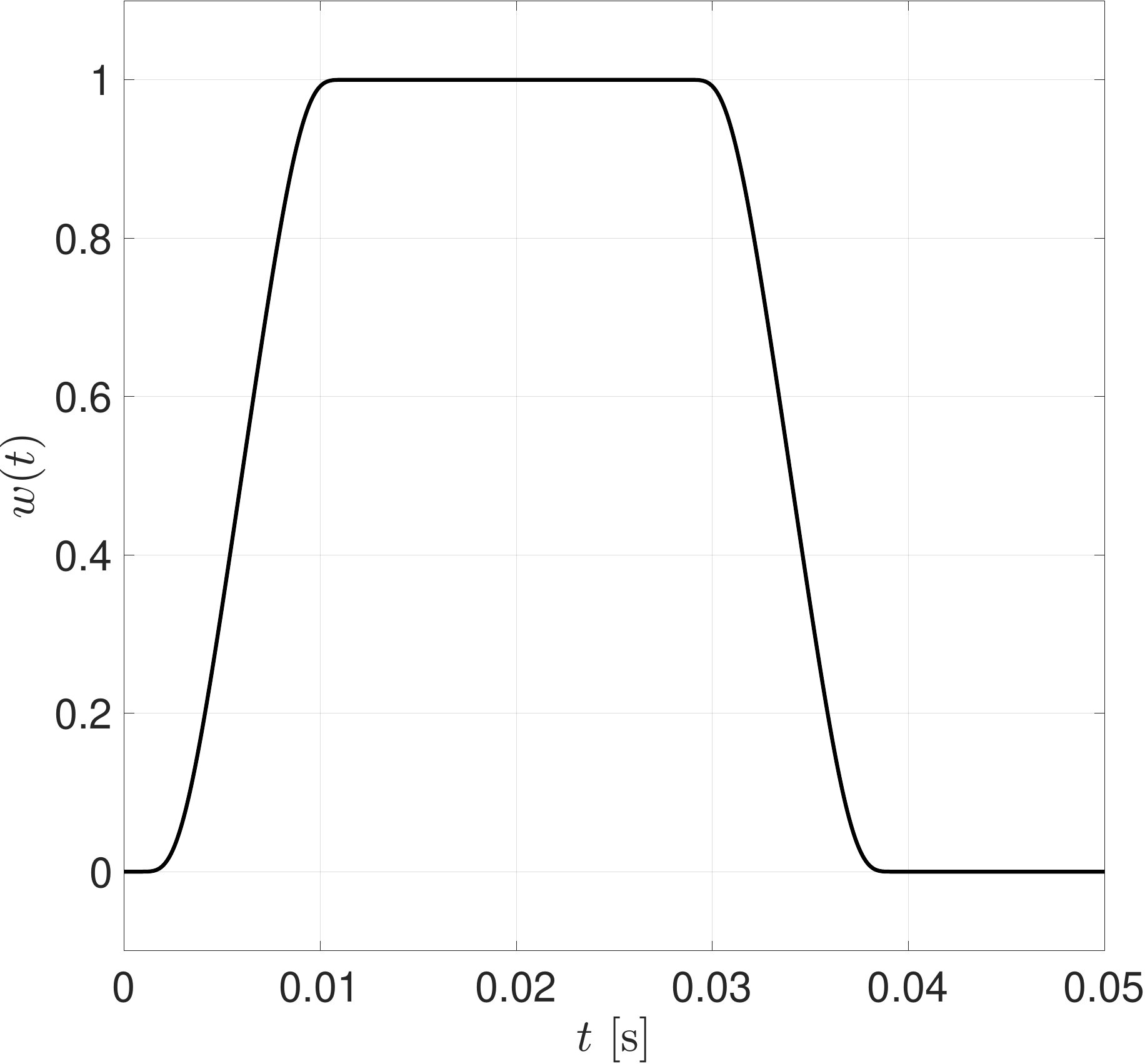}
	\caption{Time-dependent window function for the inflow boundary data of the horn.}
	\label{fig: horn_window}
\end{figure}
The optimal solution (with tolerance $10^{-6}$) is found after 111 iterations. In Figure \ref{fig: horn_snaps_300hz} the shape of the horn after 30 and 111 iterations is shown. The solution to the forward problem with the optimized shape at $t =0.01$ s, $t = 0.03$ s, $t = 0.04$ s, and $t = 0.05$ s is plotted in Figure \ref{fig: horn_300hz_forward_snaps}. Note that almost no wave propagation is present in the domain at $t = 0.05$ s. To evaluate the properties of the optimized horn the unregularized loss (given in \eqref{eq: cont_min_prob_horn_loss}) is computed for a range of input signals, one frequency at a time, and plotted in Figure \ref{fig: horn_300hz}. The results show that the optimized shape decreases the reflected sound at $300$ Hz by more than two orders of magnitude compared to the initial shape. 
\begin{figure}
\centering
\begin{subfigure}{.49\textwidth}
	\centering
	\includegraphics[width=\textwidth]{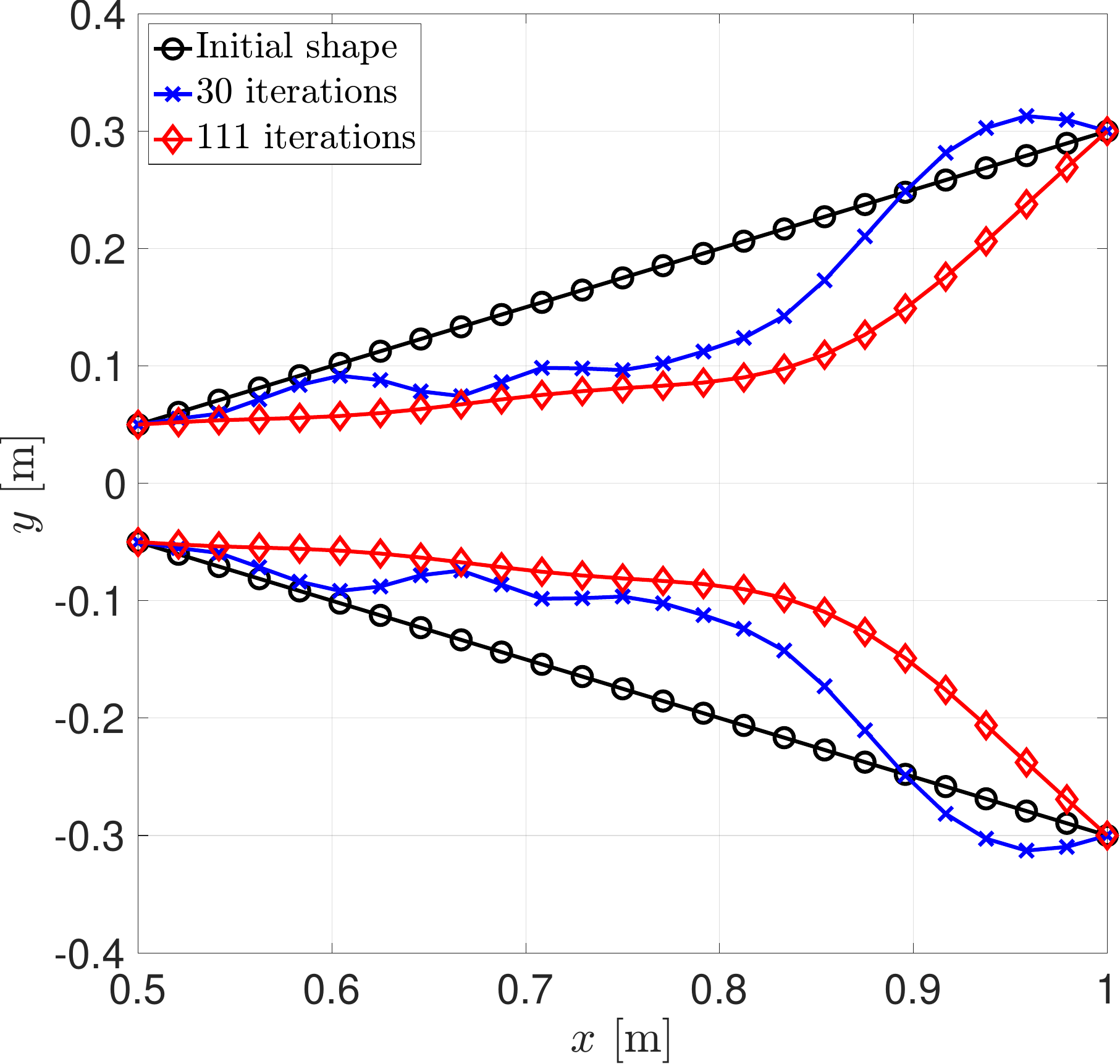}
	\caption{}
	\label{fig: horn_snaps_300hz}
\end{subfigure}
\begin{subfigure}{.49\textwidth}
	\centering
	\includegraphics[width=\textwidth]{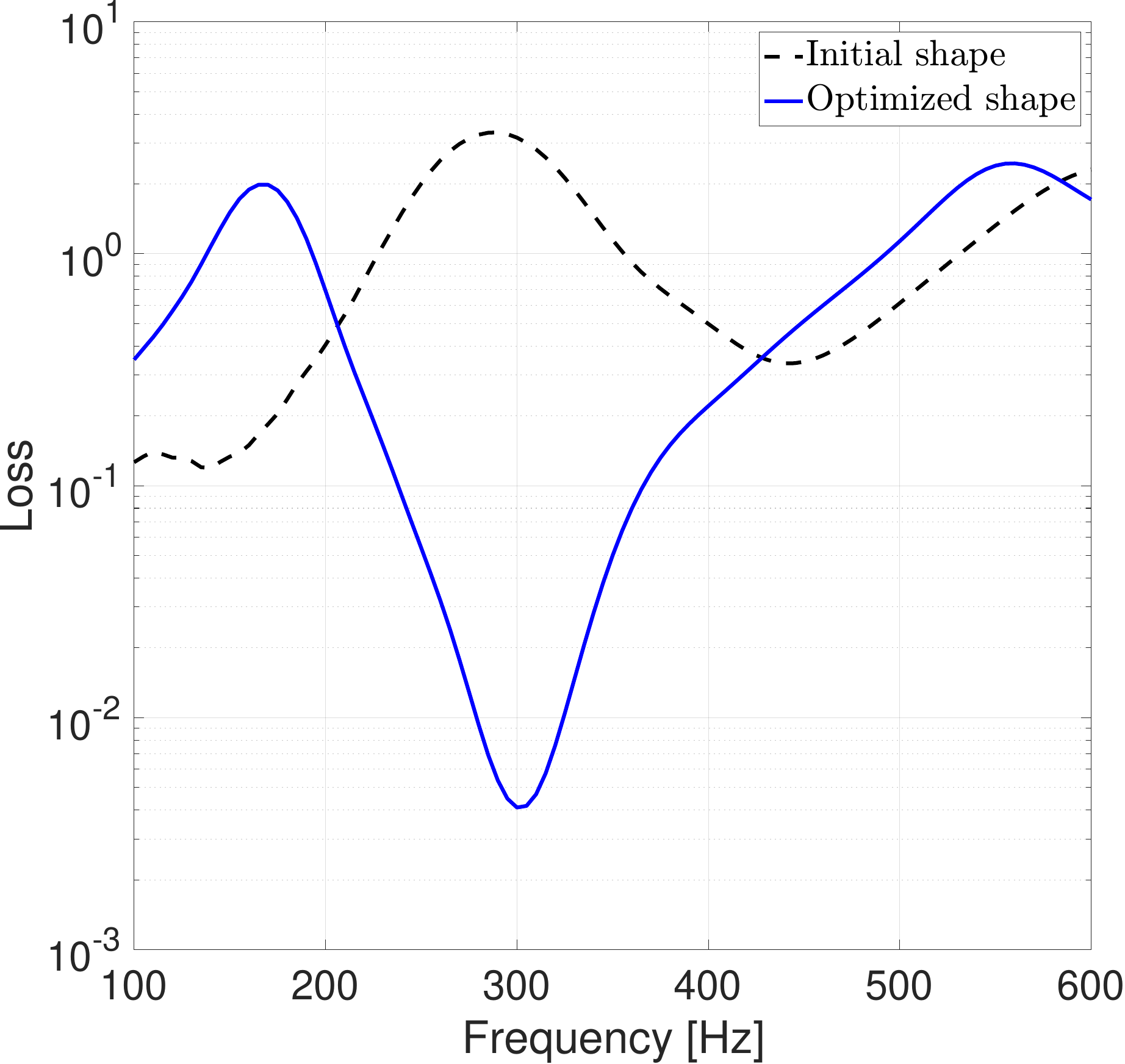}
	\caption{}
	\label{fig: horn_300hz}
\end{subfigure}
\caption{Results optimizing for 300 Hz. \textbf{(\protect\subref{fig: horn_snaps_300hz})} Shape of the horn for different optimization iterations. \textbf{(\protect\subref{fig: horn_300hz})} Unregularized loss \eqref{eq: cont_min_prob_horn_loss} for varying frequencies.}
\end{figure}
\begin{figure}[!htbp]
	\centering
	\includegraphics[width=1\textwidth]{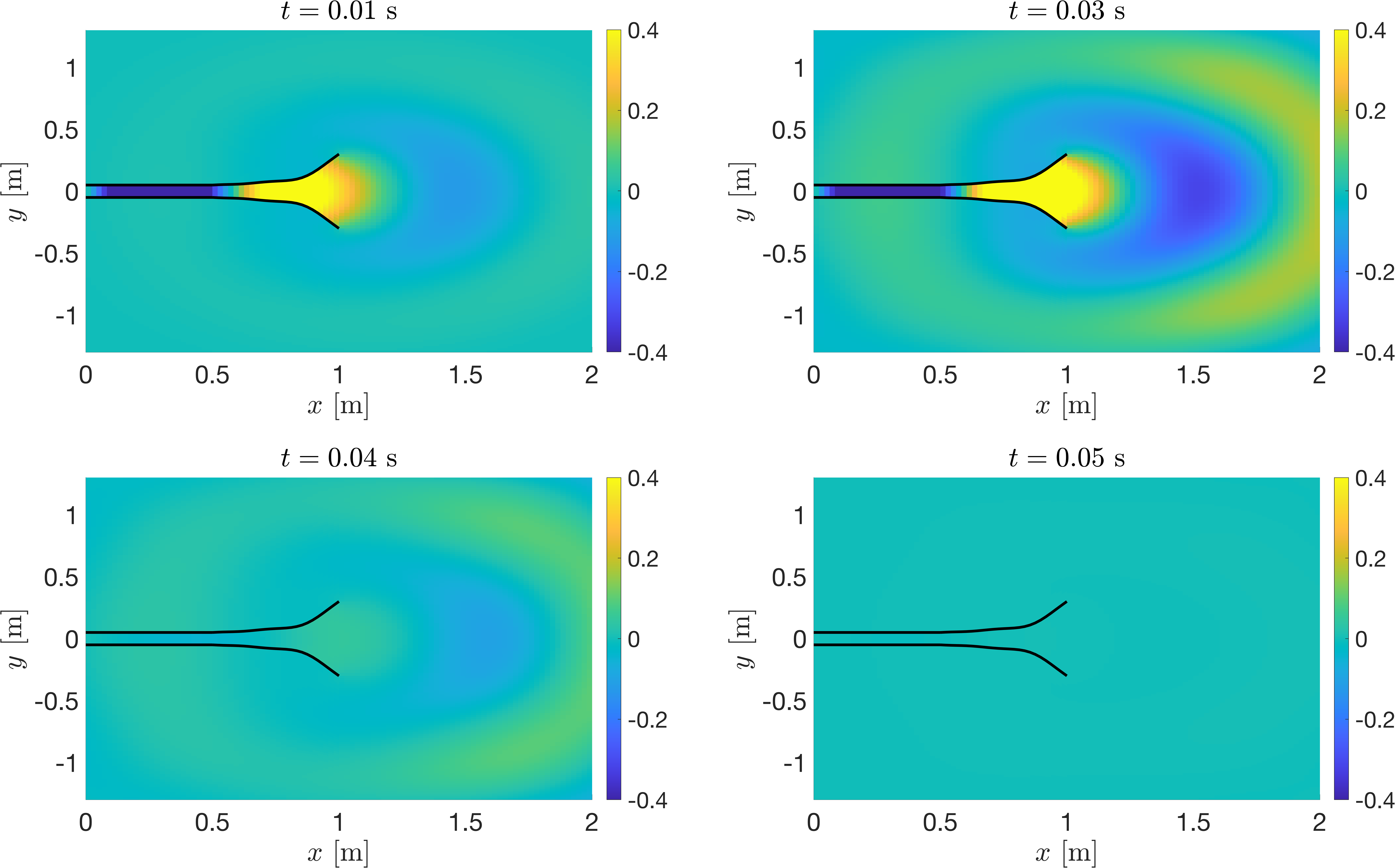}
	\caption{Snapshots of the solution to the forward problem with the air horn optimized for 300 Hz.}
	\label{fig: horn_300hz_forward_snaps}
\end{figure}

For the second experiment, we optimize the horn for a broadband signal consisting of a superposition of $N_f = 101$ equally spaced frequencies between 300 and 400 Hz. The inflow function becomes
\begin{equation}
	\label{eq: airhorn_broad_inflow}
	g(t) = \frac{w(t)}{N_f} \sum_{j = 1}^{N_f} \sin(2 \pi f_j t),
\end{equation}
where $f_j = 300 + (j-1)$, $j = 1,2,...,N_f$. Note that the computational cost of evaluating the loss and gradient with this inflow data function is almost identical to the costs of that of a single frequency (the difference lies in evaluating the sum in \eqref{eq: airhorn_broad_inflow}, which is negligible). This is one advantage of solving the problem in the time domain. The optimal solution (with tolerance $10^{-6}$) is found after 66 iterations. In Figure \ref{fig: horn_snaps_300-400hz} and \ref{fig: horn_300-400hz} the shape of the horn after 30 and 66 iterations and the unregularized loss for varying frequencies are presented, respectively. Once again we can observe that the reflected sound in the frequencies of interest is smaller by approximately two orders of magnitude for the horn with an optimized shape as compared to the initial shape.
%
%
\begin{figure}
\centering
\begin{subfigure}{.49\textwidth}
	\centering
	\includegraphics[width=\textwidth]{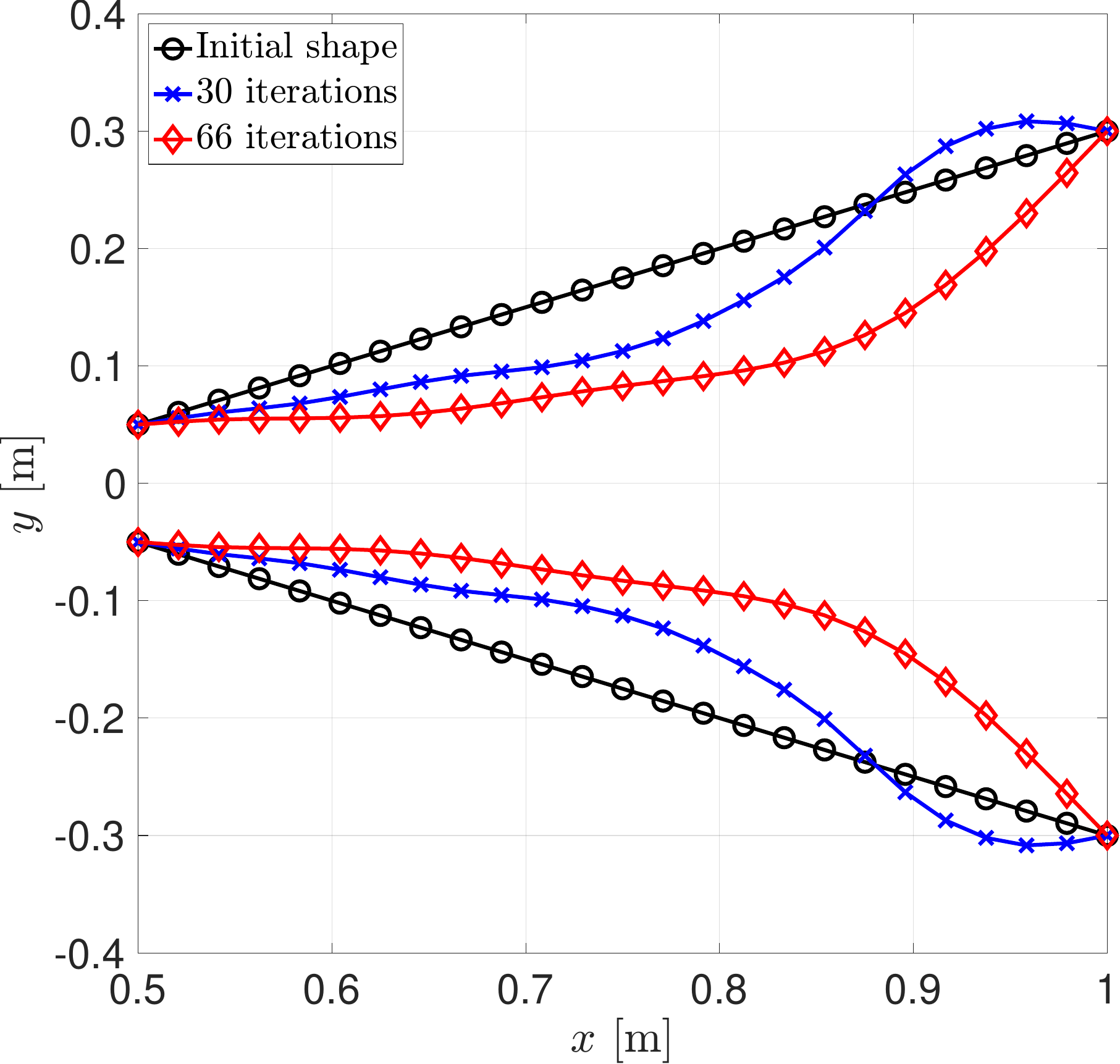}
	\caption{}
	\label{fig: horn_snaps_300-400hz}
\end{subfigure}
\begin{subfigure}{.49\textwidth}
	\centering
	\includegraphics[width=\textwidth]{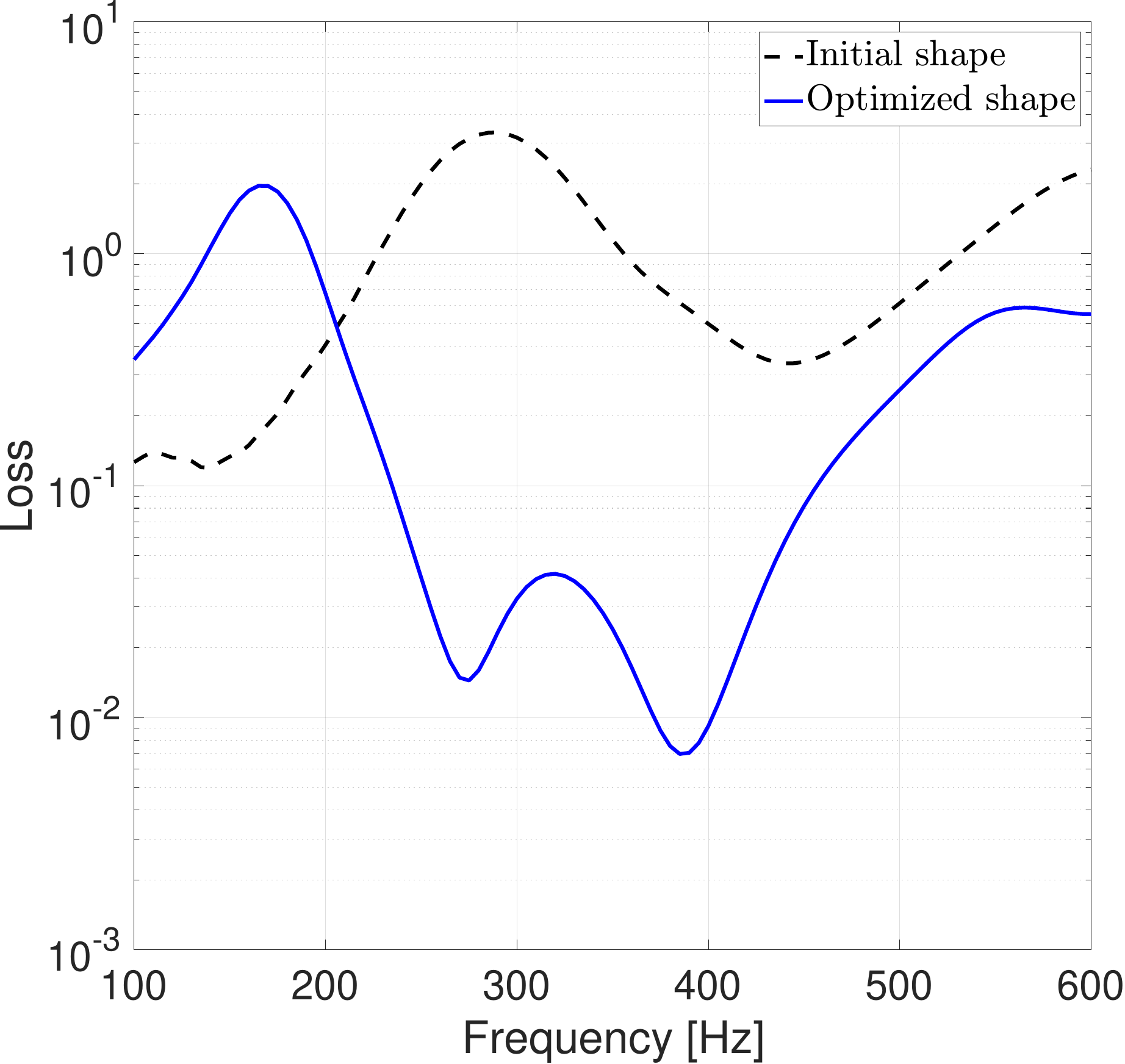}
	\caption{}
	\label{fig: horn_300-400hz}
\end{subfigure}
\caption{Results optimizing for 101 frequencies between 300 Hz and 400 Hz. \textbf{(\protect\subref{fig: horn_snaps_300-400hz})} Shape of the horn for different optimization iterations. \textbf{(\protect\subref{fig: horn_300-400hz})} Unregularized loss \eqref{eq: cont_min_prob_horn_loss} for varying frequencies.}
\label{fig: horn_300-400hz_both}
\end{figure}
}
\section{Conclusions}
\label{sec: concl}
In this paper, we present a method for solving shape optimization problems constrained by the acoustic wave equation using energy stable finite differences of high-order accuracy with SBP properties. The optimization method is based on the inversion of coordinate transformations, where the design domain is transformed to a rectangular reference domain and the resulting metric coefficients are optimized to minimize the loss functional. A gradient-based optimization algorithm is used, computing the gradient of the loss functional through the adjoint framework. From a numerical point of view, the problem is similar to full waveform inversion in seismic imaging, where the objective is to find material parameters in the medium. The problem is solved in the time domain, meaning that any time-dependent source and receiver data can be used. Using a combination of weak and strong imposition of boundary conditions through the SAT and projection methods, the scheme of the forward problem is shown to be dual consistent. 

Three numerical experiments are performed. First, the accuracy of the forward scheme is verified, demonstrating that the expected convergence rates are obtained. In the second numerical experiment, we show that the method can be used to reconstruct the shape of a seabed with only one source and one receiver close to the surface. In the final numerical experiment, we optimize the shape of the mouth of an air horn with the objective of minimizing reflected sound in the horn. For this last example, a complicated source function is used consisting of many frequencies to emphasize the advantages of solving the problem in the time domain.

Interesting topics for future research include, e.g., extensions of the method to three spatial dimensions and the impact of more sophisticated time-stepping schemes on the convergence behavior of the optimization problem. From a theoretical standpoint, the extension to 3D is straightforward, the issues lie in the efficient implementation of the solver and the computation of the gradient. Similarly, time integration using a self-adjoint method, e.g., the Runge--Kutta methods presented in \cite{sanz_serna} or SBP in time \cite{nordstrom_lundquist_2016}, is mostly an implementation challenge. {The error study of the adjoint-based gradient in Section \ref{subsec: bath} shows that choosing the time step a factor of 10 below the CFL stability limit (i.e., setting CFL constant $k = 0.1$) results in a 1\% relative error attributed to the time discretization error. Moreover, the error in the gradient seems to decrease linearly with $k$, implying that time steps close to the stability limit will result in approximately 10\% relative error. Using a self-adjoint time integrator could therefore allow for larger time steps without incurring an error. For self-adjoint Runge--Kutta methods, this entails evaluating the time integral in the adjoint-based gradient using the Runge--Kutta quadrature, thereby requiring storing also the Runge--Kutta stage values of the forward and adjoint solutions, resulting in increased memory costs \cite{sanz_serna}. Whether the benefits of obtaining the exact gradient to the fully discrete optimization problem through dual consistency in space and time outweigh the benefits in computational costs of the simpler methods remains to be answered.}

\section*{Acknowledgements}
G. Eriksson was supported by FORMAS (Grant No. 2018-00925). V. Stiernstr\"{o}m was supported by the Swedish Research Council (grant 2017-04626 VR). 

The computations were enabled by resources provided by the National Academic Infrastructure for Supercomputing in Sweden (NAISS) and the Swedish National Infrastructure for Computing (SNIC) at UPPMAX partially funded by the Swedish Research Council through grant agreements no. 2022-06725 and no. 2018-05973.

\section*{Data statement}
Data sharing not applicable to this article as no datasets were generated or analyzed during
the current study.

\appendix
\section{Operator derivatives}
\label{sec: op_dervs}
Note that all operators in $\Omega^+$ are independent of $\bv{p}$, hence we only have to consider the operators in $\Omega^-_p$. We shall begin with the derivatives of the coordinates of the grid points in $\Omega^-_p$. Using \eqref{eq: disc_xy_coords}, we get
\begin{equation}
	\label{eq: disc_coords_dp}
	\begin{alignedat}{3}
		\pder{x_{i,j}}{p_i} &= 0, \\
		\pder{y_{i,j}}{p_i} &= \bv{e}_i - \bv{e}_i \eta_j,
	\end{alignedat}
\end{equation}
where $\bv{e}_i$ is the $i$:th column of $I_{m_{\Gamma_I}}$. We also have the metric derivatives
\begin{equation}
	\label{eq: disc_metric_dervs_dp}
	\begin{alignedat}{2}
		\pder{\bv{X}_\xi}{p_i} &= \mathtt{diag}(D^-_\xi \pder{\bv{x}}{p_i}), \quad &&\pder{\bv{Y}_\xi}{p_i} = \mathtt{diag}(D^-_\xi \pder{\bv{y}}{p_i}), \\
		\pder{\bv{X}_\eta}{p_i} &= \mathtt{diag}(D^-_\eta \pder{\bv{x}}{p_i}), &&\pder{\bv{Y}_\eta}{p_i} = \mathtt{diag}(D^-_\eta \pder{\bv{y}}{p_i}),
	\end{alignedat}
\end{equation}
and metric coefficients
\begin{equation}
	\begin{alignedat}{2}
		\pder{\bv{J}}{p_i} &= \pder{\bv{X}_\xi}{p_i} \bv{Y}_{\eta} + \bv{X}_{\xi} \pder{\bv{Y}_\eta}{p_i} - \pder{\bv{X}_\eta
		}{p_i} \bv{Y}_{\xi} - \bv{X}_{\eta} \pder{\bv{Y}_\xi}{p_i}, \\
		\pder{\alpha_1}{p_i} &= -\pder{\bv{J}}{p_i} \bv{J}^{-2} (\bv{X}_\eta^2 + \bv{Y}_\eta^2) + 2 \bv{J}^{-1} (\pder{\bv{
		X}_\eta}{p_i} \bv{X}_{\eta} + \pder{\bv{Y}_\eta}{p_i} \bv{Y}_{\eta}), \\
		\pder{\beta}{p_i} &= \pder{\bv{J}}{p_i} \bv{J}^{-2} (\bv{X}_\xi \bv{X}_\eta + \bv{Y}_\xi \bv{Y}_\eta) - \bv{J}^{-
		1} (\pder{\bv{X}_\xi}{p_i} \bv{X}_\eta + \pder{\bv{Y}_\xi}{p_i} \bv{Y}_\eta + \bv{X}_\xi \pder{\bv{X}_\eta}{p_i} + \bv{Y}_\xi \pder{\bv{Y}_\eta}{p_i}), \\
		\pder{\alpha_2}{p_i} &= -\pder{\bv{J}}{p_i} \bv{J}^{-2} (\bv{X}_\xi^2 + \bv{Y}_\xi^2) + 2 \bv{J}^{-1} (\pder{\bv{X}
		_\xi}{p_i} \bv{X}_{\xi} + \pder{\bv{Y}_\xi}{p_i} \bv{Y}_{\xi}).
	\end{alignedat}
\end{equation}
Using this, we can construct the derivatives of the operators in $\Omega^-_p$ as follows
\begin{equation}
	\begin{alignedat}{2}
		\pder{D_L^-}{p_i} &= -\pder{\bv{J}}{p_i} \bv{J}^{-2} (D_{\xi \xi}^{(\alpha_1)} + D^-_\eta \beta D^-_\xi + D^-_\xi \beta D^-_\eta + D_{\eta \eta} ^{(\alpha_2)}) \\
		&+ J^{-1} (D_{\xi \xi}^{(\pder{\alpha_1}{p_i})} + D^-_\eta \pder{\beta}{p_i} D^-_\xi + D^-_\xi \pder{\beta}{p_i} D^-_\eta + D_{\eta \eta} ^{(\pder{\alpha_2}{p_i})}),
	\end{alignedat}
\end{equation}
where we have used that
\begin{equation}
	\pder{D_2^{(\bv{c}(\bv{p}))}}{p_i} = H^{-1}(-M^{(\pder{\bv{c}(\bv{p})}{p_i})} - e_l e_l^\top \pder{\bv{c}}{p_i} d_l^\top + e_r e_r^\top \pder{\bv{c}}{p_i} d_r^\top) =  D_2^{(\pder{\bv{c}(\bv{p})}{p_i})},
\end{equation}
since
\begin{equation}
	\pder{M^{(\bv{c}(\bv{p}))}}{p_i} = D_1^\top H \pder{\bar{\bv{c}}}{p_i} D_1 + R^{(\pder{\bv{c}(\bv{p})}{p_i})} = M^{(\pder{\bv{c}(\bv{p})}{p_i})},
\end{equation}
and
\begin{equation}
	\pder{R^{(\bv{c}(\bv{p}))}}{p_i} = R^{(\pder{\bv{c}(\bv{p})}{p_i})},
\end{equation}
for any $\bv{p}$-dependent vector $\bv{c}(\bv{p})$, due to the specific structure of $R^{(\bv{c}(\bv{p}))}$ (see, e.g., \cite{Mattsson2012,Almquist2020,STIERNSTROM2023112376}). We also have
\begin{equation}
	\begin{alignedat}{2}
		\pder{d^-_w}{p_i} &= e^-_w \pder{\bv{W}_2}{p_i}\bv{W}_2^{-2} (\alpha_1 e_w^{-\top} e^-_w D^-_\xi + \beta e_w^{-\top} e^-_w D^-_\eta) \\
		&-e^-_w \bv{W}_2^{-1} (\pder{\alpha_1}{p_i} e_w^{-\top} e^-_w D^-_\xi + \pder{\beta}{p_i} e_w^{-\top} e^-_w D^-_\eta), \\
		\pder{d^-_e}{p_i} &= -e^-_e \pder{\bv{W}_2}{p_i}\bv{W}_2^{-2} (\alpha_1 e_e^{-\top} e^-_e D^-_\xi + \beta e_e^{-\top} e^-_e D^-_\eta) \\
		&+ e^-_e \bv{W}_2^{-1} (\pder{\alpha_1}{p_i} e_e^{-\top} e^-_e D^-_\xi + \pder{\beta}{p_i} e_e^{-\top} e^-_e D^-_\eta), \\
		\pder{d^-_s}{p_i} &= e^-_s \pder{\bv{W}_1}{p_i} \bv{W}_1^{-2} (\alpha_2 e_s^{-\top} e^-_s D^-_\eta + \beta e_s^{-\top} e^-_s D^-_\xi) \\
		&-e^-_s \bv{W}_1^{-1} (\pder{\alpha_2}{p_i} e_s^{-\top} e^-_s D^-_\eta + \pder{\beta}{p_i} e_s^{-\top} e^-_s D^-_\xi), \\
		\pder{d^-_n}{p_i} &= -e^-_n \pder{\bv{W}_1}{p_i} \bv{W}_1^{-2} (\alpha_2 e_n^{-\top} e^-_n D^-_\eta + \beta e_n^{-\top} e^-_n D^-_\xi) \\
		&+e^-_n \bv{W}_1^{-1} (\pder{\alpha_2}{p_i} e_n^{-\top} e^-_n D^-_\eta + \pder{\beta}{p_i} e_n^{-\top} e^-_n D^-_\xi),
	\end{alignedat}
\end{equation}
with
\begin{equation}
	\pder{\bv{W}_1}{p_i} = \frac{\pder{\bv{X}_\xi}{p_i} \bv{X}_\xi + \pder{\bv{Y}_\xi}{p_i} \bv{Y}_\xi}{\sqrt{\bv{X}_\xi^2 + \bv{Y}_\xi^2}} \quad \text{and} \quad \pder{\bv{W}_2}{p_i} = \frac{\pder{\bv{X}_\eta}{p_i} \bv{X}_\eta + \pder{\bv{Y}_\eta}{p_i} \bv{Y}_\eta}{\sqrt{\bv{X}_\eta^2 + \bv{Y}_\eta^2}},
\end{equation}
and the differentiated inner product matrices
\begin{equation}
		\pder{\bar H}{p_i} = 
		\begin{bmatrix}
			0 & 0 \\ 0 & \pder{H^-}{p_i}
		\end{bmatrix}, \quad \text{with} \quad \pder{H^-}{p_i} = H^-_\xi H^-_\eta \pder{J}{p_i},
\end{equation}
and
\begin{equation}
	\begin{alignedat}{4}
		\pder{H^-_w}{p_i} &= H e_w \pder{\bv{W}_2}{p_i} e_w^\top, \quad \pder{H^-_e}{p_i} &&= H e_e \pder{\bv{W}_2}{p_i} e_e^\top, \\
		\pder{H^-_s}{p_i} &= H e_s \pder{\bv{W}_1}{p_i} e_s^\top, \quad \pder{H^-_n}{p_i} &&= H e_n \pder{\bv{W}_1}{p_i} e_n^\top.
	\end{alignedat}
\end{equation}
Using this, we get
\begin{equation}
 	\begin{alignedat}{2}
 		\pder{SAT_{BC_1}}{p_i} &= \pder{\bar H}{p_i} \bar H^{-2} \left (
		\sum_{k = w,e} \begin{bmatrix}
			e_k^{+ \top} H^{+}_k d_k^{+} & 0 \\
			0 & e_k^{- \top} H^{-}_k d_k^{-}
		\end{bmatrix} + \begin{bmatrix}
			0 & 0 \\
			0 & e_s^{- \top} H^{-}_s d_s^{-}
		\end{bmatrix} \right ) \\
		& -\bar H^{-1}
		\sum_{k = w,e} \begin{bmatrix}
			0 & 0 \\
			0 & e_k^{- \top} (\pder{H^{-}_k}{p_i} d_k^{-} + H^{-}_k \pder{d_k^{-}}{p_i})
		\end{bmatrix} \\
		& -\bar H^{-1} \begin{bmatrix}
			0 & 0 \\
			0 & e_s^{- \top} (\pder{H^{-}_s}{p_i}d_s^{-} + H^{-}_s \pder{d_s^{-}}{p_i})
		\end{bmatrix},
 	\end{alignedat}
\end{equation} 
\begin{equation}
	\begin{alignedat}{2}
		\pder{SAT_{BC_2}}{p_i} &= \pder{\bar H}{p_i} \bar H^{-2}
		\sum_{k = w,e}  \begin{bmatrix}
		 e_k^{+ \top} H^{+}_k e_k^{+} & 0 \\
		0 & e_k^{-\top} H^{-}_k e_k^{-}
	\end{bmatrix} \\
	&-\bar H^{-1}
		\sum_{k = w,e}  \begin{bmatrix}
		 0 & 0 \\
		0 & e_k^{-\top} \pder{H^{-}_k}{p_i} e_k^{-}
	\end{bmatrix},
	\end{alignedat}
\end{equation}
\begin{equation}
	\begin{alignedat}{2}
		\pder{SAT_{IC}}{p_i} &= \pder{\bar H}{p_i} \bar H^{-2}
	\begin{bmatrix}
		e_s^{+ \top} H^{+}_s d^{+}_s & e_s^{+ \top} H^{+}_s d^{-}_n \\
		0 & 0
	\end{bmatrix} \\
	&-\bar H^{-1}
	\begin{bmatrix}
		0 & e_s^{+ \top} H^{+}_s \pder{d^{-}_n}{p_i} \\
		0 & 0
	\end{bmatrix},
	\end{alignedat}
\end{equation}
and
\begin{equation}
	\begin{alignedat}{2}
		\pder{P}{p_i} &= \pder{\bar H}{p_i}\bar H^{-2} L^\top (L \bar H^{-1} L^\top)^{-1} L \\
		&- \bar H^{-1} L^\top (L \bar H^{-1} L^\top)^{-1} L \pder{\bar H}{p_i}\bar H^{-2} L^\top (L \bar H^{-1} L^\top)^{-1} L \\
		&= P \pder{\bar H}{p_i} \bar H^{-1} (I - P).
	\end{alignedat}
\end{equation}
We are now ready to write down the equations for $\pder{D}{p_i}$ and $\pder{E}{p_i}$, we have
\begin{equation}
	\begin{alignedat}{2}
	\pder{D}{p_i} &= c^2
	\pder{P}{p_i} \left ( \begin{bmatrix}
		D_L^{+} & 0 \\
		0 & D_L^{-}
	\end{bmatrix} + SAT_{BC_1} + SAT_{IC} \right ) P \\
	&+ c^2 P \left ( \begin{bmatrix}
		0 & 0 \\
		0 & \pder{D_L^{-}}{p_i}
	\end{bmatrix} + \pder{SAT_{BC_1}}{p_i} + \pder{SAT_{IC}}{p_i} \right ) P \\
	&+c^2 P \left ( \begin{bmatrix}
		D_L^{+} & 0 \\
		0 & D_L^{-}
	\end{bmatrix} + SAT_{BC_1} + SAT_{IC} \right ) \pder{P}{p_i},
	\end{alignedat}
\end{equation}
and
\begin{equation}
	\label{eq: disc_E_dp}
	\pder{E}{p_i} = c \pder{P}{p_i} SAT_{BC_2} P + c P \pder{SAT_{BC_2}}{p_i} P + c P SAT_{BC_2} \pder{P}{p_i}.
\end{equation}

\begin{remark}
	The specific form of $\pder{x_{i,j}}{p_i}$ and $\pder{y_{i,j}}{p_i}$ in \eqref{eq: disc_coords_dp} would allow us to significantly simplify the expressions in \eqref{eq: disc_metric_dervs_dp}-\eqref{eq: disc_E_dp}. But, for clarity and completeness, we present the derivatives of the operators for a general mapping.
\end{remark}

\section{Proof of Lemma \ref{lemma: cont_grad}}\label{sec: cont_grad_proof}
Start by considering the first variation of $\mathtt{J}$ in \eqref{eq: cont_min_prob_multiblock} with respect to the parameterization $p$. We have,
\begin{equation}\label{eq: variation_J}
	\begin{aligned}
	\delta \mathtt{J} 
	&= \delta\left(\frac{1}{2} \int_0^T r^2 \: dt \right) \\
	&= \int_0^T \int_{\Omega^+} r\frac{\delta r(u^+,\bv{x}_r,t)}{\delta u^+(\bv{x},t)}\delta u^+(\bv{x},t) \: d\bv{x} dt  \\
	&= \int_0^T \int_{\Omega^+} r\frac{\delta u^+(\bv{x}_r,t)}{\delta u^+(\bv{x},t)}\delta u^+(\bv{x},t) \: d\bv{x} dt  \\
	&= \int_0^T \int_{\Omega^+} r \hat{\delta}(\bv{x} - \bv{x}_r)\delta u^+(\bv{x},t) \: d\bv{x} dt  \\
	&= \int_0^T(r\hat{\delta}(\bv{x} - \bv{x}_r),\delta u^+)_{\Omega^+} \: dt,
  \end{aligned}
\end{equation}
where $r = r(u^+,\bv{x}_r,t) = u^+(\bv{x}_r,t) - u_d(t)$. The Lagrangian loss functional of \eqref{eq: cont_min_prob_multiblock} is 
\begin{equation}\label{eq: lagrangian_cont_multiblock}
	\mathtt{L} = \mathtt{J} + \int_0^T (\lambda^+, u^+_{tt} - c^2\Delta u^+ - f(t)\hat{\delta}(\bv{x} - \bv{x}_s)_{\Omega^+} + (\lambda^-, u^-_{tt} - c^2\Delta u^-))_{\Omega^-} \: dt.
\end{equation}
Note that $\mathtt{L} = \mathtt{J}$ such that $\delta \mathtt{L} = \delta \mathtt{J}$ for any $u^\pm$ satisfying \eqref{eq: cont_wave_eq_multiblock}, and we therefore proceed with analyzing $\delta \mathtt{L}$. As a first step \eqref{eq: lagrangian_cont_multiblock} is recast into a suitable form through integration by parts, starting with the contribution from $\Omega^+$. Note that $\delta f(t) \hat{\delta}(\bv{x} - \bv{x_s}) = 0$ (since $\Omega^+$ is independent of $p$) and we therefore disregard the point source in the following analysis. Integrating the second term in \eqref{eq: lagrangian_cont_multiblock} by parts in time and space twice results in
\begin{equation}
	\int_0^T (\lambda^+, u^+_{tt} - c^2\Delta u^+ )_{\Omega^+} \: dt = VT^+ + IT^+,
\end{equation}
where
\begin{equation}\label{eq: VTp}
	VT^+ = \int_0^T (\lambda^+_{tt} - c^2\Delta \lambda^+, u^+)_{\Omega^+} \: dt,
\end{equation}
are the volume terms and 
\begin{equation}\label{eq: ITp}
	IT^+ = \int_0^T - \langle\lambda^+, c^2\bv{n}^+ \cdot \nabla u^+ \rangle_{\delta \Omega^{(+,s)}} + \langle c^2\bv{n}^+ \cdot \nabla \lambda^+, u^+ \rangle_{\delta \Omega^{(+,s)}} \: dt,
\end{equation}
are the interface terms. Any boundary terms vanish by the forward and adjoint initial and boundary conditions in \eqref{eq: cont_wave_eq_multiblock}, \eqref{eq: cont_adj_wave_eq_multiblock}. Due to the $p$-dependency of $\Omega_p^-$, the third term in \eqref{eq: lagrangian_cont_multiblock} requires a slightly different treatment. Integrating in time twice and space once and using the forward and adjoint initial and boundary conditions yields results in
\begin{equation}
	\int_0^T (\lambda^-, u^-_{tt} - c^2\Delta u^- )_{\Omega^-_p} \: dt = VT^- + BT^- + IT^-,
\end{equation}
where 
\begin{equation}\label{eq: VTm}
	VT^- = \int_0^T  (\lambda^-_{tt}, u^-)_{\Omega^-_p} + (\nabla \lambda^-, c^2\nabla u^-)_{\Omega^-_p} \: dt,
\end{equation}

\begin{equation}\label{eq: BTm}
	\begin{aligned}
	BT^- &= \int_0^T \langle c\lambda^-_t, u^-\rangle_{\partial \Omega^{(-,e)}_p} + \langle c\lambda^-_t, u^-\rangle_{\partial \Omega^{(-,w)}_p} \: dt,
	\end{aligned}
\end{equation}
and
\begin{equation}\label{eq: ITm}
	IT^- = \int_0^T -\langle\lambda^-, c^2\bv{n}^- \cdot \nabla u^-\rangle_{\partial \Omega^{(-,n)}_p} \: dt.
\end{equation}
$BT^-$ is obtained by using the boundary conditions $c\bv{n}^- \cdot \nabla u^- = u^-_t$ and integrating by parts in time, canceling boundary terms using the forward and adjoint initial conditions. Adding the interface terms \eqref{eq: ITp} and \eqref{eq: ITm} (exchanging $\del \Omega^{(+,s)}$, $\del \Omega_p^{(-,n)}$ for $\Gamma_I$) yields
\begin{equation}\label{eq: IT}
	\begin{aligned}
	IT &= IT^+ + IT^- \\
	=\int_0^T \Big( -&\langle\lambda^+, c^2\bv{n}^+ \cdot \nabla u^+\rangle_{\Gamma_I} +\langle c^2\bv{n}^+ \cdot \nabla \lambda^+, u^+\rangle_{\Gamma_I} \\
	- & \langle\lambda^-, c^2\bv{n}^- \cdot \nabla u^-\rangle_ {\Gamma_I} \Big)\:  dt\\
	= \int_0^T &\langle c^2\bv{n}^+ \cdot \nabla \lambda^+, u^+\rangle_{\Gamma_I} \: dt,
	\end{aligned}
\end{equation}
where the last equality was obtained using the forward and adjoint interface conditions in \eqref{eq: cont_wave_eq_multiblock}, \eqref{eq: cont_adj_wave_eq_multiblock}. To summarize, we have arrived at
\begin{equation}
	\mathtt{L} = \mathtt{J} + VT^+ + VT^- + BT^- + IT,
\end{equation}
and will now consider $\delta \mathtt{L}$. 

To start with, by \eqref{eq: VTp} $\delta VT^+ = \int_0^T(\lambda^+_{tt} - c^2\Delta \lambda^+, \delta u^+)_{\Omega^+} \: dt$ such that
\begin{equation}\label{eq: variation_1}
	\delta (\mathtt{J} + VT^+) = \int_0^T(\lambda^+_{tt} - c^2\Delta \lambda^+ + r \hat{\delta}(\bv{x} - \bv{x_r}), \delta u^+)_{\Omega^+ } \: dt = 0,
\end{equation}
due to \eqref{eq: variation_J} and \eqref{eq: cont_adj_wave_eq_multiblock}. To obtain $\delta VT^-$, we first transform \eqref{eq: VTm} to the reference domain $\tilde{\Omega}^-$ using \eqref{eq: cont_curv_dx_dy}, \eqref{eq: cont_curv_metric_coeff} and \eqref{eq: cont_curv_norms}, resulting in
\begin{equation}
	\begin{aligned}
	VT^- = \int_0^T \Big((\lambda^-_{tt}, Ju^-)_{\tilde{\Omega}^-} &+ c^2(\lambda^-_\xi,\alpha_1 u^-_\xi + \beta u^-_\eta)_{\tilde{\Omega}^-} \\
	&+ c^2(\lambda^-_\eta,\alpha_2 u^-_\eta + \beta u^-_\xi)_{\tilde{\Omega}^-} \Big) \: dt.
	\end{aligned}
\end{equation}
Taking the first variation, applying the product rule and grouping terms containing $\delta u^-$ and variations of metric terms (using that $\delta u^-_{\xi,\eta} = (\delta u^-)_{\xi,\eta}$), results in
\begin{equation}
	\begin{aligned}
	&\delta VT^- = VT^-_{\delta u} + VT^-_{\delta X},
	\end{aligned}
\end{equation}
with 
\begin{equation}\label{eq: VTm_du}
	\begin{aligned}
	VT^-_{\delta u} = \int_0^T \Big((\lambda^-_{tt}, J \delta u^-)_{\tilde{\Omega}^-} &+ c^2(\lambda^-_\xi, \alpha_1 (\delta u^-)_\xi + \beta (\delta u^-)_\eta)_{\tilde{\Omega}^-}\\
	&+ c^2(\lambda^-_\eta, \alpha_2 (\delta u^-)_\eta + \beta (\delta u^-)_\xi)_{\tilde{\Omega}^-} \Big) \: dt,
	\end{aligned}
\end{equation}
and 
\begin{equation}\label{eq: VTm_dX}
	\begin{aligned}
	VT^-_{\delta X} = \int_0^T \Big((\lambda^-_{tt}, \delta J u^-)_{\tilde{\Omega}^-} &+ c^2(\lambda^-_\xi, \delta \alpha_1 u^-_\xi + \delta \beta u^-_\eta)_{\tilde{\Omega}^-}\\
	&+ c^2(\lambda^-_\eta, \delta \alpha_2 u^-_\eta + \delta \beta u^-_\xi)_{\tilde{\Omega}^-} \Big) \: dt.
	\end{aligned}
\end{equation}
Integrating $VT^-_{\delta u}$ in \eqref{eq: VTm_du} by parts, raising $\delta u^-$, and transforming back to the physical domain yields
\begin{equation}
	\begin{aligned}
	VT^-_{\delta u} &= \int_0^T(\lambda^-_{tt} - c^2\Delta \lambda^-,  \delta u^-)_{\Omega^-_p} + \langle c^2\bv{n}^- \cdot \nabla \lambda^-, \delta u^- \rangle_{\partial\Omega^{-}_p} \: dt\\
	&= BT^-_{\delta u} + IT^-_{\delta u},
	\end{aligned}
\end{equation}
where \eqref{eq: cont_adj_wave_eq_multiblock} is used to cancel the volume and boundary terms such that the remaining terms are
\begin{equation}\label{eq: BTm_du}
	\begin{aligned}
	BT^-_{\delta u} &= \int_0^T \langle c^2\bv{n}^- \cdot \nabla \lambda^-, \delta u^- \rangle_{\Omega^{(-,w)}_p} + \langle c^2\bv{n}^- \cdot \nabla \lambda^-, \delta u^- \rangle_{\Omega^{(-,e)}_p} \: dt,
	\end{aligned}
\end{equation}
and
\begin{equation}\label{eq: ITm_du}
	\begin{aligned}
	IT^-_{\delta u} &=  \int_0^T\langle c^2\bv{n}^- \cdot \nabla \lambda^-, \delta u^-\rangle_{\Gamma_I} \: dt.
	\end{aligned}
\end{equation}
For $VT^-_{\delta X}$ in \eqref{eq: VTm_dX} we integrate by parts, raising $\lambda^-$, resulting in
\begin{equation}
	\begin{aligned}
	VT^-_{\delta X} &= V^-_{VT} + BT^-_{\delta X},
	\end{aligned}
\end{equation}
where
\begin{equation}\label{eq: V_VT}
	V^-_{VT} = \int_0^T c^2(\lambda^-, \delta J \Delta u^- - (\delta \alpha_1 u^-_\xi + \delta \beta u^-_\eta)_\xi - (\delta \alpha_2 u^-_\eta + \delta \beta u^-_\xi)_\eta)_{\tilde{\Omega}^-} \: dt,
\end{equation}
and
\begin{equation}\label{eq: BT_dX}
	\begin{aligned}
	BT^-_{\delta X} = \int_0^T c^2\Big( &\langle\lambda^-, \delta\alpha_1u^-_\xi + \delta\beta u^-_\eta\rangle_{\partial\tilde{\Omega}^{(-,e)}} \\
	-&\langle\lambda^-, \delta\alpha_1u^-_\xi + \delta\beta u^-_\eta\rangle_{\partial\tilde{\Omega}^{(-,w)}} \\
	+&\langle \lambda^-, \delta\alpha_2u^-_\eta + \delta\beta u^-_\xi \rangle_{\del \tilde{\Omega}^{(-,n)}}\\
	-&\langle\lambda^-, \delta\alpha_2u^-_\eta + \delta\beta u^-_\xi\rangle_{\partial\tilde{\Omega}^{(-,s)}}\Big) \: dt.
	\end{aligned}
\end{equation}
To obtain \eqref{eq: V_VT} we used $u^-_{tt} = c^2\Delta u^-$ (where $\Delta u^-$ is given by \eqref{eq: curv_laplace}) to substitute the second derivative in time. Focusing on the remaining boundary terms, we return to \eqref{eq: BTm} and transform to the reference domain
\begin{equation}
	BT^- = \int_0^t \langle\lambda^-_t, cW_1 u^- \rangle_{\partial \tilde{\Omega}^{(-,e)}} + \langle\lambda^-_t, cW_1 u^- \rangle_{\partial \tilde{\Omega}^{(-,w)}} \: dt.
\end{equation}
By the product rule $\delta (W_1 u^-) =  \delta W_1 u^- + W_1 \delta u^-$, it follows that
\begin{equation}\label{eq: dBTm}
	\begin{aligned}
		\delta BT^- = \int_0^T \Big( & \langle\lambda^-_t, c\delta W_1 u^-\rangle_{\partial \tilde{\Omega}^{(-,e)}} + \langle\lambda^-_t, c\delta W_1 u^-\rangle_{\partial \tilde{\Omega}^{(-,w)}} \\
		 + & \langle c\lambda^-_t, \delta u^-\rangle_{\partial\Omega^{(-,e)}} + \langle c\lambda^-_t, \delta u^- \rangle_{\partial\Omega^{(-,w)}} \Big) \: dt,	
	\end{aligned}
\end{equation}
where the two latter terms are transformed back to the physical domain. Considering $BT^-_{\delta u} + \delta BT^- $ (given by \eqref{eq: BTm_du}, \eqref{eq: dBTm}) the terms on the physical domain cancel since by the adjoint outflow conditions in \eqref{eq: cont_adj_wave_eq_multiblock}, $\langle\lambda^-_t + c\bv{n}^- \cdot \nabla \lambda^-, \delta u^-\rangle_{\partial\Omega^{(-,w)}} = 0$ due to $\lambda_\tau = -\lambda_t$ (and similar for the east boundary). Additionally including the term $BT^-_{\delta X}$ in \eqref{eq: BT_dX}, the variation of the boundary terms are
\begin{equation}\label{eq: V_BT}
\begin{aligned}
	V^-_{BT} = &BT^-_{\delta X} + BT^-_{\delta u} + \delta BT^- =\\
		\int_0^T \Big(&c^2(\langle \lambda^-, \delta\alpha_1u^-_\xi + \delta\beta u^-_\eta\rangle_{\partial\tilde{\Omega}^{(-,e)}} - \langle\lambda^-, \delta\alpha_1u^-_\xi + \delta\beta u^-_\eta\rangle_{\partial\tilde{\Omega}^{(-,w)}})\\ 
		+&c^2(\langle\lambda^-, \delta\alpha_2u^-_\eta + \delta\beta u^-_\xi\rangle_{\del \tilde{\Omega}^{(-,n)}} - \langle\lambda^-, \delta\alpha_2u^-_\eta + \delta\beta u^-_\xi\rangle_{\partial\tilde{\Omega}^{(-,s)}}) \\
		+&c(\langle \lambda^-_t, \delta W_1 u^- \rangle_{\partial \tilde{\Omega}^{(-,w)}} + \langle \lambda^-_t, \delta W_1 u^- \rangle_{\partial \tilde{\Omega}^{(-,e)}}) \Big) \: dt.
\end{aligned}
\end{equation}

For the variation of interface terms, first note that by \eqref{eq: IT} $\delta IT = (c^2\bv{n}^+ \cdot \nabla \lambda^+, \delta u^+)_{\Gamma_I}$. Then by \eqref{eq: ITm_du}, $\delta IT + IT^-_{\delta u} = 0$, due to the forward and adjoint interface conditions of \eqref{eq: cont_wave_eq_multiblock} and \eqref{eq: cont_adj_wave_eq_multiblock}. We have therefore arrived at $\delta \mathtt{L} = V^-_{VT} + V^-_{BT}$, given by \eqref{eq: V_VT}, and \eqref{eq: V_BT}. Grouping terms in $V^-_{VT}$, and $V^-_{BT}$ by inner products with $\lambda^-$ and $\lambda_t^-$ and transforming back to the physical domain results in
\begin{equation}\label{eq: cont_variation}
	\delta \mathtt{L} =  V_{\lambda} + V_{\lambda_t},
\end{equation}
where
\begin{equation}\label{eq: V_lambda}
	\begin{aligned}
	V_{\lambda} = -c^2\Big( -&(\lambda^-, J^{-1}\delta J\Delta u^-)_{\Omega^- }\\
	+&(\lambda^-,J^{-1}((\delta \alpha_1 u^-_\xi + \delta \beta u^-_\eta)_\xi + (\delta \alpha_2 u^-_\eta + \delta \beta u^-_\xi)_\eta))_{\Omega^- } \\
	-&\langle \lambda^-, W_1^{-1}(\delta \alpha_1u^-_\xi + \delta\beta u^-_\eta) \rangle_{\partial\Omega^{(-,e)} } \\
	+&\langle \lambda^-,  W_1^{-1}(\delta \alpha_1u^-_\xi + \delta\beta u^-_\eta) \rangle_{\partial\Omega^{(-,w)} } \\
	-&\langle\lambda^-,  W_2^{-1}(\delta\alpha_2u^-_\eta + \delta\beta u^-_\xi )\rangle_{\partial\Omega^{(-,n)}} \\
	+&\langle\lambda^-, W_2^{-1}(\delta\alpha_2u^-_\eta + \delta\beta u^-_\xi)\rangle_{\partial\Omega^{(-,s)} } \Big),
	\end{aligned}
\end{equation}
and
\begin{equation}\label{eq: V_lambda_t}
	\begin{aligned}
		V_{\lambda_t} = &c\Big(\langle \lambda^-_t, W_1^{-1}\delta W_1 u^- \rangle_{\partial\Omega^{(-,e)}} + \langle \lambda^-_t, W_1^{-1}\delta W_1 u^- \rangle_{\partial\Omega^{(-,w)}}\Big).
	\end{aligned}
\end{equation}
Since $\delta\mathtt{J} = \delta\mathtt{L} = V_{\lambda} + V_{\lambda_t}$, we arrive at $\frac{\delta \mathtt{J}}{\delta p} = \frac{\delta \mathtt{L}}{\delta p} = G_{\lambda} + G_{\lambda_t}$, i.e, \eqref{eq: cont_grad}. This completes the proof.

\clearpage
\bibliographystyle{elsarticle-num}
\bibliography{references}

\end{document}